\documentclass[a4paper]{amsart}
\usepackage[english]{babel}
\usepackage{amsmath,amssymb,amsthm}

\usepackage[pdftex]{color}
\usepackage{graphicx}
\usepackage{xcolor}
\usepackage[bookmarks=true,hyperindex,pdftex,colorlinks, citecolor=blue,linkcolor=blue, urlcolor=blue]{hyperref}

\hypersetup{pdftitle={The Bishop-Phelps-Bollob\'as property for numerical radius},
pdfauthor={Sun Kwang Kim, Han Ju Lee, Miguel Mart\'{\i}n and \'Oscar, Rold\'an},pdfpagemode=UseOutlines}

\theoremstyle{plain}
\newtheorem{theorem}{Theorem}[section]
\newtheorem{corollary}[theorem]{Corollary}

\newtheorem{proposition}[theorem]{Proposition}
\newtheorem{lemma}[theorem]{Lemma}

\theoremstyle{definition}

\newtheorem{remark}[theorem]{Remark}
\newtheorem{example}[theorem]{Example}

\newtheorem{definition}[theorem]{Definition}
\newtheorem{question}[theorem]{Question}

\DeclareMathOperator{\re}{Re\,}

\DeclareMathOperator{\Id}{\mathrm{Id}}

\DeclareMathOperator{\sign}{sign}
\DeclareMathOperator{\aconv}{aconv}

\newcommand{\K}{\mathbb{K}}

\newcommand{\C}{\mathbb{C}}

\newcommand{\N}{\mathbb{N}}

\newcommand{\eps}{\varepsilon}

\newcommand{\call}{\mathcal{L}}

\newcommand{\bbk}{\mathbb{K}}

\newcommand{\na}{\operatorname{NA}}
\newcommand{\nra}{\operatorname{NRA}}

\renewcommand{\leq}{\leqslant}
\renewcommand{\geq}{\geqslant}

\begin{document}

\title{The Bishop-Phelps-Bollob\'as property for the numerical radius: a Zizler-type approach}

\author[Kim]{Sun Kwang Kim}
\address[Kim]{Department of Mathematics, Chungbuk National University, 1 Chungdae-ro, Seowon-Gu, Cheongju, Chungbuk 28644, Republic of Korea}
\email{skk@chungbuk.ac.kr}
\urladdr{\href{http://orcid.org/0000-0002-9402-2002}{ORCID: \texttt{0000-0002-9402-2002} } }

\author[Lee]{Han Ju Lee}
\address[Lee]{Department of Mathematics Education, Dongguk University - Seoul, 04620 (Seoul), Republic of Korea}
\email{hanjulee@dongguk.edu}
\urladdr{
\href{http://orcid.org/0000-0001-9523-2987}{ORCID: \texttt{0000-0001-9523-2987} } }

\author[Mart\'{\i}n]{Miguel Mart\'{\i}n}
\address[Mart\'{\i}n]{Department of Mathematical Analysis and Institute of Mathematics (IMAG), University of Granada, E-18071 Granada, Spain}
\email{mmartins@ugr.es}
\urladdr{\url{https://www.ugr.es/local/mmartins/}}
\urladdr{
\href{http://orcid.org/0000-0003-4502-798X}{ORCID: \texttt{0000-0003-4502-798X} } }

\author[Rold\'an]{\'Oscar Rold\'an}
\address[Rold\'an]{Departamento de An\'{a}lisis Matem\'{a}tico, Universidad de Valencia, Avenida Vicente Andr\'{e}s Estell\'{e}s 19, 46100 Burjasot (Valencia), Spain}
\email{oscar.roldan@uv.es}
\urladdr{\href{http://orcid.org/0000-0002-1966-1330}{ORCID: \texttt{0000-0002-1966-1330} } }



\thanks{ }

\begin{abstract}
We investigate the Bishop–Phelps–Bollob\'as property for the numerical radius (BPBp-nu) through a Zizler-type perspective on the classical Bishop–Phelps–Bollob\'as property (BPBp). This approach allows us to establish two new results: the real Banach space $\ell_\infty$ satisfies the BPBp-nu, while the complex space $\ell_1 \oplus_\infty c_0$ does not. Note that the latter provides the first natural example —constructed without renorming techniques— of a Banach space where the numerical radius attaining operators are dense but the BPBp-nu fails. Along the way, we strengthen the main results of \cite{KLM16} concerning the interplay between the BPBp for the pair $(X,Y)$ and the BPBp-nu for a direct sum $X\oplus Y$ of Banach spaces. We further explore the validity of the Zizler-type BPBp across different pairs of Banach spaces, and how this property relates to the classical BPBp and the BPBp-nu. Finally, we specialize our analysis to the framework of compact operators.
\end{abstract}

\date{\today}

\maketitle

\section{Introduction}\label{section:intro}

\subsection{Preliminaries and notation}
We use standard notation on functional analysis, see for instance \cite{FHHMZ11}. Throughout the whole paper, $X$ and $Y$ are Banach spaces over the field $\mathbb{K}=\mathbb{R}$ or $\mathbb{C}$. Let $X^*$, $B_X$, and $S_X$ respectively denote the topological dual, closed unit ball, and unit sphere of $X$. We denote by $\mathcal{L}(X,Y)$ the space of all bounded and linear operators from $X$ to $Y$, and by $\mathcal{K}(X,Y)$ its subspace consisting of all compact operators from $X$ to $Y$. If $X=Y$, we just write $\call(X)$, and the same applies to other sets of mappings. Recall that an operator $T\in\mathcal{L}(X,Y)$ \textit{attains its norm} if there exists some $x_0\in S_X$ such that $\|T(x_0)\|=\|T\|$. The set of all operators from $X$ into $Y$ that attain their norms is denoted by $\na(X,Y)$. It is worth noting that the Hahn-Banach theorem assures that $\na(X,\K)$ is not empty and, moreover, that $\na(X,Y)$ is also not empty (by just defining convenient rank-one norm attaining operators). 

In 1961, Bishop and Phelps proved the famous \emph{Bishop-Phelps theorem} which asserts that for every Banach space $X$, the set $\na(X,\bbk)$ is always dense in $X^*$ (see for instance \cite[Theorem 7.41]{FHHMZ11}), and they left as an open question whether $\na(X,Y)$ is always dense in $\call(X,Y)$. In 1963, Lindenstrauss answered that question in the negative in his seminal paper \cite{Lindenstrauss63}. He also showed that this density holds in some classes of Banach spaces, such as when $X$ is reflexive or when $Y$ satisfies a geometric property known as property $\beta$, which is satisfied for instance by spaces $Y$ such that $c_0\subset Y\subset \ell_\infty$ (canonical copies) and by finite-dimensional polyhedral spaces. This initiated a new research line to find out for which pairs of Banach spaces $\na(X,Y)$ is dense in $\call(X,Y)$. We refer the reader to the expository papers \cite{Acosta06, ChoiKimLee-survey2024, Martin-RACSAM2016} and the recent papers \cite{Bachir, CJ23, Fovelle,JMR-JFA2023,JMzR-2024} for more information and background. Among interesting results on the topic, Zizler showed that for all $X$ and $Y$, the set of operators whose adjoints are norm attaining is always dense in $\call(X,Y)$ (see \cite{Zizler73}); Bourgain and Huff showed that $X$ has the Radon-Nikod\'ym property (RNP for short) if and only if for every $Z$ isomorphic to $X$ and every $Y$, $\na(Z,Y)$ is dense in $\call(Z,Y)$ (see \cite{Bourgain77,Huff80}); and nevertheless the density holds for most pairs of classical Banach spaces, but not for all of them. For instance, Schachermayer showed that $\na(L_1[0,1], C[0,1])$ is not dense in $\call(L_1[0,1], C[0,1])$ (see \cite{Schachermayer83}). We also remark that, while there are compact operators which cannot be approximated by norm attaining ones (see \cite{martinjfa}), the question of whether finite-rank operators can be always approximated by norm attaining ones remains open, even for the shocking case when the target space is the two-dimensional Hilbert space. 

In 1970, Bollob\'as in \cite{Bollobas70} gave a quantitative refinement of the Bishop-Phelps theorem, showing that a functional and a point at which it almost attains its norm can be approximated simultaneously by a pair for which the norm is attained. This result is now known as the \emph{Bishop-Phelps-Bollob\'as theorem}.

\begin{theorem}[\mbox{Bishop-Phelps-Bollob\'as theorem, see \cite[Corollary 2.4]{C-K-M-M-R} for this version}] For a Banach space $X$ and $0<\eps<2$, if $x_0\in B_X$ and $x_0^*\in B_{X^*}$ satisfy $\re x_0^*(x_0)>1-\frac{\eps^2}{2}$, then there exist $x_1\in S_X$ and $x_1^*\in S_{X^*}$ such that $x_1^*(x_1)=1$, $\|x_1^*-x_0^*\|<\eps$, and $\|x_1-x_0\|<\eps$.
\end{theorem}

In 2008, Acosta, Aron, Garc\'{\i}a, and Maestre introduced the following vector-valued analogue of the Bishop-Phelps-Bollob\'as theorem (see \cite{AAGM08}).

\begin{definition}[{\cite{AAGM08}}]
A pair \((X, Y)\) of Banach spaces is said to have the \textit{Bishop-Phelps-Bollob\'as property for operators} (\textit{BPBp} for short) if, for every \(\varepsilon > 0\), there exists \(\eta(\varepsilon) > 0\) such that whenever \(T \in \mathcal{L}(X,Y)\) with \(\|T\| = 1\) and \(x_0 \in S_X\) satisfy 
$$\|T x_0\| > 1 - \eta(\varepsilon),$$ there exist \(x_1 \in S_X\) and \(S \in \mathcal{L}(X,Y)\) such that
\[
\|S\| = \|S x_1\| = 1, \quad \|x_1 - x_0\| < \varepsilon, \quad \text{and} \quad \|S-T\| < \varepsilon.
\]
\end{definition}

It is immediate that this property is equivalent to its analogue defined by taking $x_0 \in B_X$ and $T \in B_{\mathcal{L}(X,Y)}$, provided that the function $\eta$ is appropriately adjusted. This property is satisfied, for instance, by the following pairs of Banach spaces $(X,Y)$: when $X$ and $Y$ are finite-dimensional (note that the finite dimensionality of $X$ alone is not sufficient) (see \cite{AAGM08}); when $X$ is uniformly convex (see \cite{KL14a}); when $Y$ has the aforementioned property $\beta$ of Lindenstrauss (see \cite{AAGM08}); when $X=L_p(\mu)$ and $Y=L_q(\nu)$ for arbitrary measures $\mu,\nu$ and real numbers $p,q\in [1,\infty)$ (see \cite{CKLM14}); when $X=C(K_1)$ and $Y=C(K_2)$ for arbitrary compact Hausdorff spaces $K_1$ and $K_2$  in the real case (see \cite{ABCCKLLM14}); when $X=C_0(L)$ and $Y=L_q(\nu)$ for arbitrary locally compact Hausdorff space $L$, measure $\nu$ and  $1< q<\infty$ (proved for $C(K)$ in \cite{KL15} and extendable to $C_0(L)$ following the argument in \cite{Acosta16}). In \cite{Acosta16}, it is also shown that the BPBp holds for pairs $(C_0(L), L_1(\nu))$ for arbitrary locally compact Hausdorff space $L$ and measure $\nu$ in the complex case, but it is open whether even $(c_0, \ell_1)$ satisfies the property in the real case.

 It was shown in \cite{AAGM08} that $(\ell_1,Y)$ satisfies the BPBp if and only if $Y$ satisfies a geometric property known as the \textit{Approximate Hyperplane Series Property}. 

\begin{definition}[{\cite{AAGM08}}]\label{AHSp1}
A Banach space $X$ is said to have the \textit{Approximate Hyperplane Series Property} (\textit{AHSp} for short) if, for every \(\varepsilon > 0\),  there is $\eta(\eps)>0$ such that whenever $\{x_k\}_{k\in \mathbb{N}}\subset B_X$, $x^*\in S_Z$, and $a=(a_k)_{k\in \mathbb{N}}\in S_{\ell_1}$ with $a_k\geq 0$ for all $k\in\mathbb{N}$ satisfy
$$
\re x^*\left( \sum_{k=1}^{\infty} a_k x_k \right)>1-\eta(\eps),
$$
there are $A\subset \mathbb{N}$, $\{z_k\}_{k\in A}\subset S_X$ and $z^*\in S_{Z}$ such that
\begin{enumerate}
\item[(i)] $\sum_{k\in A} a_k>1-\eps$,
\item[(ii)] $\|z_k-x_k\|<\eps$ for all $k\in A$,
\item[(iii)] $z^*(z_k)=1$ for all $k\in A$.
\end{enumerate}
\end{definition}

The AHSp and its consequences have been extensively studied in the literature, and it is known to be satisfied by most, if not all, classical Banach spaces. It is worth noting that if $(L_1(\mu), Y)$ has the BPBp, then $Y$ must have the AHSp, and if $Y$ has both the AHSp and the RNP, then $(L_1(\mu), Y)$ has the BPBp (see \cite{CKLM14}). Stability results under sums of spaces, and versions of the BPBp for specific classes of operators (for instance when all involved operators are compact) have also been considered in the literature, see for instance \cite{ACKLM15, CDJM19, DGMM18}. We refer to the surveys \cite{Acosta19,DGMR22} and references therein for further background about this property and related ones.

For $x\in X$, we define the sets 
\begin{align*}
D(X, x)&=\{x^*\in S_{X^*}\colon x^*(x)=\|x\|\}\ \ \text{~and~}\ \ 
\Pi(X)= \{ (x, x^*) \in S_X\times S_{X^*}: x^*\in D(X,x)\}.
\end{align*}

For an operator $T\in \mathcal{L}(X)$, the \emph{numerical radius} $\nu(T)$ of $T$ is given by
$$\nu(T) =\sup \{ |x^*Tx|\colon  (x, x^*) \in \Pi(X)\},$$
and the \emph{numerical index} of the space $X$ is the number
$$n(X)=\inf \{\nu(T)\colon T\in\call(X,X),\, \|T\|=1\}.$$

From its definition, the numerical index $n(X)$ of $X$ satisfies $0 \leq n(X) \leq 1$.  Note that the $n(X) = 1$ means the numerical radius $\nu$ coincides with the norm, and this holds, for example, for $L_1(\mu)$ spaces and their isometric preduals—hence for $C_0(L)$ spaces—among many others. For background on the numerical radius and numerical index, we refer, for instance, to the survey paper \cite{KaMaPa}, to \cite[Section 1.1]{KLMM20}, and to the book \cite{spearbook} and the references therein.  

An operator $T\in \mathcal{L}(X)$ is said to \textit{attain its numerical radius} if there exists $(x_0, x_0^*)\in \Pi(X)$ such that 
$$\nu(T) = |x_0^*Tx_0|.$$

We denote by $\nra(X)$ the set of all numerical radius attaining operators on $X$. Sims asked in his PhD dissertation when $\nra(X)$ is dense in $\call(X)$. This question was then studied by several authors, and it was systematically addressed in several papers from the 1980s and 1990s, see for instance \cite{AcoPaya89, AcoPaya93}. We note that there exists compact operators which cannot be approximated by numerical radius attaining ones (see \cite{CMM17}), but $\nra(X)$ is dense whenever $X$ has the RNP (see \cite{AcoPaya93}). In 2013, Guirao and Kozhushkina \cite{GK13} introduced a Bishop-Phelps-Bollob\'as version of the density of numerical radius attaining operators. We recall the following two versions, one from the already cited \cite{GK13} and the second from \cite{KLM14}.

\begin{definition}[\textrm{\cite{GK13,KLM14}}]~
\begin{itemize}
\item A Banach space $X$ is said to have the \textit{Bishop-Phelps-Bollob\'as property for numerical radius} (\textit{BPBp-nu} for short) if, for every $0<\eps<1$, there exists $\eta(\eps)>0$ such that
whenever $T\in \mathcal{L}(X)$ and $(x_0, x_0^*)\in \Pi(X)$ satisfy  
$$\nu(T)=1 \text{~and~} |x_0^*Tx_0|>1-\eta(\eps),$$
there exist $S\in \mathcal{L}(X)$ and $(x_1, x_1^*)\in \Pi(X)$ such that
\[
\nu(S) = |x_1^*Sx_1|=1, \ \ \|S-T\|<\eps, \ \ \|x_1-x_0\|<\eps,\ \ \text{and} \ \ \|x_1^*- x_0^*\|<\eps.
\]

\item A Banach space $X$ is said to have the \textit{weak Bishop-Phelps-Bollob\'as property for numerical radius} (\textit{weak BPBp-nu} for short) if, for every $0<\eps<1$, there exists $\eta(\eps)>0$ such that
whenever $T\in \mathcal{L}(X)$ and $(x_0, x_0^*)\in \Pi(X)$ satisfy 
 $$\nu(T)=1 \text{~and~} |x_0^*Tx_0|>1-\eta(\eps),$$
there exist $S\in \mathcal{L}(X)$ and $(x_1, x_1^*)\in \Pi(X)$ such that
\[
\nu(S) = |x_1^*Sx_1|, \ \ \|S-T\|<\eps, \ \ \|x_1-x_0\|<\eps,\ \ \text{and} \ \ \|x_1^*- x_0^*\|<\eps.
\]
\end{itemize}
\end{definition}

The difference between the weak BPBp-nu and the BPBp-nu lies in whether or not the operator $S$ is required to have numerical radius 1. Due to the difficulties in approximating elements in $X$ and $X^*$ simultaneously, little is known about the BPBp-nu compared to the BPBp. It is shown that all finite-dimensional spaces, $c_0$, and all $L_p(\mu)$ spaces with $1 \leq p < \infty$ for any measure $\mu$ have the BPBp-nu (see \cite{GK13, KLM14, KLMM20}). For a compact Hausdorff space $K$, it is an open question whether $C(K)$ has the BPBp-nu (see \cite[Question 9]{DGMR22} and \cite[Section 4.3-(a)]{AGR14}). It  is only known positively for a few special spaces such as a metric space $K$ (see \cite{AGR14}). In fact, it was open whether even the real space $\ell_\infty = C(\beta \mathbb{N})$ satisfies the BPBp-nu where $\beta \mathbb{N}$ is the Stone-\v{C}ech compactification of the natural numbers, and we prove that it is true in Corollary \ref{cor:real-ell-infty-BPBp-nu}. We remark that, if one considers the analogous property to the BPBp-nu restricted to compact operators, then all $C(K)$ spaces are known to satisfy this version in both the real and complex cases (see \cite{GMMR21}). On the other hand, it is worth noting that every infinite-dimensional separable Banach space can be equivalently renormed to fail the BPBp-nu (see \cite[Theorem~17]{KLM14}), even though $\nra(X)$ is dense in those spaces with the RNP. Further examples of Banach spaces failing the BPBp-nu can be derived from the main results in \cite{KLM16}, which we will improve in Section \ref{section:bpbzp-and-bpbpnu} of this paper. For more background on this topic, we refer the reader to \cite[Section~2.7]{DGMR22} and the references therein. 

This paper aims to explore the BPBp-nu by applying the following Zizler-type perspective to the BPBp. 

\begin{definition}
A pair $(X, Y)$ of Banach spaces is said to have the \textit{Bishop-Phelps-Bollob\'as-Zizler property} (\textit{BPBZp} for short) if, for every $\varepsilon > 0$, there exists $\eta(\varepsilon) > 0$ such that whenever $T \in \mathcal{L}(X,Y)$ with $\|T\| = 1$, $y_0^* \in S_{Y^*}$, and $x_0 \in S_X$ satisfy 
\[
|y_0^*(T x_0)| > 1 - \eta(\varepsilon),
\]
there exist $y_1^* \in S_{Y^*}$, $x_1 \in S_X$, and $S \in \mathcal{L}(X,Y)$ such that 
\[
\|S\| = |y_1^*(S x_1)| = 1, \quad \|x_1 - x_0\| < \varepsilon, \quad \|y_1^* - y_0^*\| < \varepsilon, \quad \text{and} \quad \|S-T\| < \varepsilon.
\]
\end{definition}

Note that the BPBZp clearly implies the BPBp by the Hahn-Banach theorem. The BPBZp was first introduced, albeit without a name, in \cite{CKLM15} as a tool to obtain a BPB-type version of Zizler's result on the density of operators whose adjoints attain their norms. It was shown in \cite[Proposition 3.5]{CKLM15} that $(X, c_0)$ satisfies the BPBZp if and only if $X$ has the AHSp for the pair $(X, X^*)$. 

\begin{definition}(\cite{ABGM13})\label{AHSppair}
A Banach space $X$ is said to have the \textit{AHSp for the pair}  $(X,X^*)$ if, for every $\eps>0$, there is $\eta(\eps)>0$ such that whenever $\{x_k^*\}_{k\in \mathbb{N}}\subset B_{X^*}$, $x\in S_X$, and $a=(a_k)_{k\in \mathbb{N}}\in S_{\ell_1}$ with $a_k\geq 0$ for all $k\in\mathbb{N}$ satisfy
$$
\re \left( \sum_{k=1}^{\infty} a_k x_k^* \right)(x)>1-\eta(\eps),
$$
there are $A\subset \mathbb{N}$, $\{z_k^*\}_{k\in A}\subset S_{X^*}$ and $z\in S_{X}$ such that
\begin{enumerate}
\item[(i)] $\sum_{k\in A} a_k>1-\eps$,
\item[(ii)] $\|z_k^*-x_k^*\|<\eps$ for all $k\in A$,
\item[(iii)] $z_k^*(z)=1$ for all $k\in A$,
\item[(iv)] $\|z-x\|<\eps$.
\end{enumerate}
\end{definition}

The AHSp for the pair  $(X,X^*)$ was introduced in \cite{ABGM13} in order to characterize $X$ such that $(\ell_1 \times X, \mathbb{K})$ has the BPBp for bilinear forms. We omit the definition of the BPBp for bilinear forms, but we just remark that, whenever $Y$ is reflexive, $(X\times Y^*, \bbk)$ has the BPBp for bilinear forms if and only if $(X,Y)$ has the BPBZp by their definitions. 

\subsection{Outline of the paper}

The rest of the paper is organized as follows. In Section \ref{section:bpbzp-and-bpbpnu}, we investigate the relationship between the (weak) BPBp-nu for $X \oplus_1 Y$ or $X \oplus_\infty Y$ and the BPBZp for $(X, Y)$. In Theorems \ref{thm1} and \ref{thm2}, we present substantial improvements to the main results in \cite{KLM16}, from which it follows, in particular, that every pair of the form $(L_1(\mu), L_1(\nu))$ and many pairs of the form $(C(K_1), C(K_2))$ in the real case satisfy the BPBZp.

In Section \ref{section-BPBZp-pair-classical}, we study the BPBZp for pairs of Banach spaces so that at least one of them is a classical Banach space. We introduce a new AHSp-type property that generalizes the corresponding property for a pair of Definition \ref{AHSppair}, and we use it to characterize when $(\ell_1, Y)$ satisfies the BPBZp (see Theorem \ref{l1charac}). As a consequence, we conclude that the complex space $X = \ell_1 \oplus_\infty c_0$ does not satisfy the BPBp-nu (in fact, it also fails the BPBp-nu for compact operators), even though $\nra(X)$ is dense in $\call(X)$ (see Proposition \ref{prop:nra-dense-complex-l1c0}). To the authors' knowledge, this is the first natural example of such a space not involving renormings. We also provide necessary and sufficient conditions for the pairs $(L_1(\mu), Y)$ and $(X, \ell_\infty)$ to satisfy the BPBZp, along with further results involving $C_0(L)$ and $L_p(\mu)$ spaces.

In Section \ref{section:sufficient-condition-BPBpnu}, we introduce a new condition, named property (nu), which allows us to relate the BPBZp for the pair $(X, X)$ with the BPBp-nu for $X$ whenever $n(X) = 1$. We show, for instance, that all $L_1(\mu)$ spaces, all Hilbert spaces, $c_0$, and all real $C_0(L)$ spaces satisfy this property. As a consequence, we conclude that the real space $\ell_\infty$ satisfies the BPBp-nu (see Corollary \ref{cor:real-ell-infty-BPBp-nu}), a result not covered by previous works. In \cite{AGR14}, the authors asked whether a compact Hausdorff space $K$ admits local compensation if $C(K)$ has the BPBp-nu. Note that the answer to this question would be negative if $\beta \mathbb{N}$ does not admit local compensation (note however that this is currently unknown). 

In Section \ref{section:cpt}, we study the properties considered in the previous sections for the setting of compact operators. In \cite{martinjfa} and \cite{CMM17}, it was respectively shown that there exist compact operators which cannot be approximated by norm-attaining or by numerical radius attaining operators, answering long-standing open problems in the field. A key element in proving these two facts was the use of spaces without the approximation property. However, in the case of the BPBZp and the BPBp-nu for compact operators, examples of spaces without these properties can be constructed even without using a space that lacks the approximation property.

Regarding the various BPBp-type properties discussed throughout this paper, it is sufficient to consider small values of $\varepsilon$. Consequently, when proving that a given pair or space has a BPBp-type property, we shall assume without loss of generality that $0 < \varepsilon < 1$.

\section{BPBZp and BPBp-nu}\label{section:bpbzp-and-bpbpnu}
We begin by providing examples of pairs of Banach spaces having the BPBZp. Our first result demonstrates that this property holds for any pair of finite-dimensional spaces.

\begin{proposition}
If $X$ and $Y$ are  finite dimensional Banach spaces, then the pair $(X,Y)$ has the BPBZp.
\end{proposition}

\begin{proof} Otherwise, there exists $\eps_0>0$ which does not satisfy the condition in the definition of BPBZp. Thus, for each $n\in\mathbb{N}$ there exist $T_n\in S_{\mathcal{L}(X,Y)}$ and $(u_n,v^*_n)\in S_X\times S_{Y^*}$ with $|v_n^*(T_nu_n)|>1-\frac{1}{n}$ such that if $S\in S_{\mathcal{L}(X,Y)}$ and $(x_1,y^*_1)\in S_X\times S_{Y^*}$ satisfy $|y_1^*(Sx_1)|=1$, then $\max\bigl\{\|x_1 - u_n\|, \|y_1^*-v_n^*\|, \|S-T_n\|\bigr\}\geq \eps_0$. From compactness, by passing to subsequences, we can assume that $T_n$, $z_n$ and $z^*_n$ converge to $R$, $u$, and $v^*$ respectively. Since $R\in S_{\mathcal{L}(X,Y)}$, $(u,v^*)\in S_X\times S_{Y^*}$ and $|v^*(Ru)|=1$, we get a contradiction.
\end{proof}

We next show that the uniform smoothness of the range allows to pass from the BPBp to the BPBZp.

\begin{proposition}\label{uniformlysmoothBPBp}
Suppose that $X$ is a Banach space and $Y$ is a uniformly smooth Banach space. If the pair $(X,Y)$ has the BPBp, then it has the BPBZp.
\end{proposition}

\begin{proof}
Since $Y$ is uniformly smooth, its dual $Y^*$ is uniformly convex. Let $\delta_{Y^*}$ denote the modulus of convexity of $Y^*$. For the duality between uniform smoothness and uniform convexity, as well as the definition and basic properties of the modulus of convexity, we refer the reader to \cite[Chapter 9]{FHHMZ11}.

We shall show that if $(X,Y)$ has the BPBp with a function $\eta$, then it has the BPBZp with the function 
$$\gamma\colon \eps \mapsto \min\left\{\frac{\delta_{Y^*}(\eps)}{2}, \eta\left( \min\left\{\frac{\delta_{Y^*}(\eps)}{2},\eps\right\}\right)\right\}$$ for $0<\eps<1$.

Define $\eps' = \min\left\{\frac{\delta_{Y^*}(\eps)}{2},\eps\right\}$, and assume $T\in \mathcal{L}(X, Y)$ with $\|T\|=1$ and $(x_0,y^*_0)\in S_X\times S_{Y^*}$ satisfy 
$$|y_0^*(Tx_0)|>1-\gamma(\eps).$$
Since $\|Tx_0\|>1-\eta(\eps')$, there exist $x_1\in S_X$ and $S\in \mathcal{L}(X,Y)$ satisfying 
$\|Sx_1\|=1=\|S\|$ and $\max\bigl\{\|x_1 - x_0\|, \|S-T\|\bigr\} <\eps'$. For $z_1^*\in S_{Y^*}$ such that $z_1^*(Sx_1)=1$ , we have that 
$$\re z_1^*(Tx_0)\geq \re z_1^*(Sx_1)-\|S-T\|-\|x_1-x_0\|>1-2\eps'.$$
Hence, for $\theta\in \mathbb{K}$ with modulus $1$ such that $|y_0^*(Tx_0)|=\theta\,y_0^*(Tx_0)$, we see that 
$$\left\|\frac{z_1^*+\theta y_0^*}{2}\right\|\geq\re\left(\frac{z_1^*+\theta y_0^*}{2}\right)(Tx_0)>1-\frac{1}{2}\gamma(\eps)-\eps'\geq1-\delta_{Y^*}(\eps).$$
Therefore, we have $\left\|y_0^*-\overline{\theta}z_1^*\right\|<\eps$, and so $y_1^*=\overline{\theta} z_1^*$ is the desired functional.
\end{proof}

Recall that if $X$ is uniformly convex, then $(X, Y)$ has the BPBp for every Banach space $Y$ (see \cite{KL14a}). It was further shown   that all pairs of the form $(L_1(\mu), L_q(\nu))$ have the BPBp whenever $1 < q < \infty$ for arbitrary measures $\mu$ and $\nu$ (see \cite{CKLM14}). Later, it was proved in \cite{Acosta16}  that if $L$ is a locally compact Hausdorff space and $Y$ is $\mathbb{C}$-uniformly convex, then the pair $(C_0(L), Y)$ of complex Banach spaces has the BPBp. The proof also applies to the real case when $Y$ is uniformly convex, leading to the following result.

\begin{corollary}\label{uniformlysmoothBPBpcor}
For the following Banach spaces $X$ and $Y$, the pair $(X,Y)$ has the BPBZp.
\begin{enumerate}
\item $X=L_p(\mu)$ and $Y=L_q(\nu)$ for $1\leq p \leq \infty$, $1<q<\infty$, and arbitrary measures $\mu$ and $\nu$.
\item $X=C_0(L)$ and $Y=L_q(\nu)$ for $1<q<\infty$, arbitrary measure $\nu$, and arbitrary locally compact Hausdorff space $L$.
\end{enumerate}
\end{corollary}

Our next goal is to investigate the BPBZp for the pair $(L_1(\mu), L_1(\nu))$. To this end, we require the following theorem, which strengthens \cite[Theorem 2.1]{KLM16}.

\begin{theorem}\label{thm1}
Let $X$ and $Y$ be Banach spaces and suppose that $n(X)=1$. If $X\oplus_1 Y$ has the weak BPBp-nu with a function $\eta$, then $(X, Y)$ has the BPBZp with the function $\gamma\colon \eps\mapsto\eta\left(\frac{\eps}{4+\eps}\right)$ for $0<\eps<1$. 
\end{theorem}

\begin{proof} Let $P_X$ and $P_Y$ be the canonical projections from $X\oplus_1 Y$ to $X$ and from $X\oplus_1 Y$ to $Y$ respectively, and $E_X$ and $E_Y$ be the canonical injections from $X$ to $X\oplus_1 Y$ and from $Y$ to $X\oplus_1 Y$ respectively.

For $0<\eps<1$, define $\eps' = \frac{\eps}{4+\eps}$ and assume $T\in \mathcal{L}(X, Y)$ with $\|T\|=1$ and $(x_0,y^*_0)\in S_X\times S_{Y^*}$ satisfy 
$$|y_0^*(Tx_0)|>1-\eta(\eps').$$

For $E_Y\circ T \circ P_X \in \mathcal{L}$($X\oplus_1 Y$), we see that
\begin{align*}
\nu(E_Y\circ T \circ P_X) &\leq \|E_Y\circ T \circ P_X\|=\|T\|= \sup \left\{|y^*(Tx)|\colon  x\in S_X,~y^*\in S_{Y^*}\right\} \\ 
& = \sup \left\{ \bigl|(x^*,y^*)(E_Y\circ T \circ P_X(x,0))\bigr| \colon (x,x^*)\in \Pi(X),~y^*\in S_{Y^*}\right\} \\ 
&=\sup\left\{\bigl|(x^*,y^*)(E_Y\circ T \circ P_X(x,0))\bigr|\colon  \bigl((x,0),(x^*,y^*)\bigr)\in\Pi(X\oplus_1 Y)\right\}\\
&\leq \nu(E_Y\circ T \circ P_X).
\end{align*}

Therefore, $\nu(E_Y\circ T \circ P_X)=\|E_Y\circ T \circ P_X\|=\|T\|=1$. Fix $x_0^*\in S_{X^*}$ such that $(x_0,x_0^*)\in \Pi(X)$, then it holds that
$$\bigl((x_0,0),(x_0^*,y_0^*)\bigr)\in \Pi(X\oplus_1 Y) \text{~and~}
\bigl|(x^*_0, y_0^*)(E_Y\circ T \circ P_X(x_0, 0))\bigr| = |y_0^*(T x_0)|>1-\eta(\eps').
$$

Hence, there exist $\bigl((x_1, y_1),(x_1^*, y_1^*)\bigr)\in \Pi(X\oplus_1 Y)$ and $ R\in \mathcal{L}(X\oplus_1 Y)$ with
$\nu(R)=\bigl|(x_1^*, y^*_1)R(x_1, y_1)\bigr|$, $\| R - E_Y\circ T \circ P_X\| <\eps'$, $\|x_1-x_0\|+\|y_1\|<\eps'$ and $\max\bigl\{ \|x_1^* - x_0^*\|, \|y_1^* - y_0^*\|\bigr\} <\eps'$.

Since $|\nu(R)-\nu(E_Y\circ T \circ P_X)|\leq \|R-E_Y\circ T \circ P_X\|<\eps'<1/5$ and $\nu(E_Y\circ T \circ P_X)=1=\|E_Y\circ T \circ P_X\|$, the operator $Q = {\nu(R)^{-1}}R\in\mathcal{L}(X\oplus_1 Y)$ is well defined, and it is clear that $\nu(Q) =1= \bigl|(x_1^*, y^*_1)Q(x_1, y_1)\bigr|$. Moreover, we have
\begin{align*}
\| Q - E_Y\circ T \circ P_X\|&\leq \| Q - R\|+\|R-E_Y\circ T \circ P_X\|\\
&< \|R\|\nu(R)^{-1}|1-\nu(R)|+ \eps' <(1+\eps')(1-\eps')^{-1}\eps'+\eps'=\frac{\eps}{2}.
\end{align*}

At this moment, we see that $y_1=0$ which implies $x_1^*(x_1)=\|x_1^*\|=\|x_1\|=1$. This is followed by the easy observation
$\bigl|(x^*, y^*) U(x,0)\bigr|\leq \nu(U)\|x\|$ for any $\bigl((x, y),(x^*, y^*)\bigr)\in \Pi(X\oplus_1 Y)$ and $U\in L(X\oplus_1 Y)$.
Indeed, if $\|y_1\|\neq 0$, then
\begin{align*}
1&= \bigl|(x_1^*, y^*_1) Q(x_1, y_1)\bigr|\leq \bigl|(x_1^*, y^*_1) Q(x_1,0)\bigr|+ \left|(x_1^*, y^*_1) Q\left(0, \frac{y_1}{\|y_1\|}\right)\right|\|y_1\|\\
&\leq \nu(Q)\|x_1\|+ \left|(x_1^*, y^*_1) (Q-E_Y\circ T \circ P_X)\left(0, \frac{y_1}{\|y_1\|}\right)\right|\|y_1\|\\
&\leq \|x_1\|+\frac{\eps}{2}\|y_1\|<\|x_1\|+\|y_1\|=1
\end{align*}
which is the contradiction. This leads us to have $1= \nu(Q) = |(x_1^*, y_1^*)(Q(x_1, 0))|$.

We now see that $\|Q\circ E_X\|=\sup\{ \|P_X\circ Q\circ E_Xx\| +\|P_Y \circ Q\circ E_X x\| \colon x\in B_X\}=1$. Since we have that 
\begin{align*}
1= \nu(Q) &= |(x_1^*, y_1^*)Q(x_1, 0)|= \bigl|x_1^*(P_X\circ Q\circ E_Xx_1) + y^*_1(P_Y\circ Q\circ E_Xx_1)\bigr| \\
&\leq \|P_X\circ Q\circ E_Xx_1\| + \|P_Y\circ Q\circ E_Xx_1\| \leq \|Q\circ E_X\|,
\end{align*}
it is enough to show that $\|P_X\circ Q\circ E_Xp\| +\|P_Y\circ Q\circ E_Xp\|\leq 1$ for any fixed $p\in S_X$. If $P_X\circ Q\circ E_Xp=0$, we have
\begin{align*}
\|P_Y \circ Q\circ E_Xp\|&=\sup\left\{\bigl|(p^*,y^*)(Q(p,0))\bigr|\colon  (p,p^*)\in \Pi(X),~ y^*\in S_{Y^*}\right\}\\
&\leq \sup\left\{\bigl|(x^*,y^*)(Q(x,0))\bigr|\colon  (x,x^*)\in\Pi(X),~ y^*\in S_{Y^*}\right\}\\
&=\sup\left\{\bigl|(x^*,y^*)(Q(x,0))\bigr|\colon \bigl((x,0),(x^*,y^*)\bigr)\in\Pi(X\oplus_1 Y)\right\}\leq \nu(Q)=1.
\end{align*}
If $P_X\circ Q\circ E_Xp\neq 0$, take $q^*\in S_{Y^*}$ such that $q^*(P_Y \circ Q\circ E_Xp)=\|P_Y \circ Q\circ E_Xp\|$ and define $H_p\in \mathcal{L}(X)$ by 
$$H_px=P_X\circ Q\circ E_Xx+\frac{P_X\circ Q\circ E_Xp}{\|P_X\circ Q\circ E_Xp\|}q^*(P_Y \circ Q\circ E_Xx)\qquad (x\in X).$$
Then, from the assumption $n(X)=1$ we have
\begin{align*}
&\|P_X\circ Q\circ E_Xp\| +\|P_Y \circ Q\circ E_Xp\| =\|H_p p\| \leq \|H_p\|= \nu(H_p)\\
&= \sup \left\{ \bigl|x^*(P_X\circ Q\circ E_Xx) + \frac{x^*(P_X\circ Q\circ E_Xp)}{\|P_X\circ Q\circ E_Xp\|} q^*(P_Y \circ Q\circ E_Xx)\bigr| \colon (x,x^*)\in \Pi(X)\right\} \\ 
&\leq \sup \left\{ |x^*(P_X\circ Q\circ E_Xx)| + |y^*(P_Y \circ Q\circ E_Xx)| \colon (x,x^*)\in \Pi(X),~y^*\in S_{Y^*}\right\} \\ 
& = \sup \left\{ \bigl|(x^*,y^*)(Q(x,0))\bigr| \colon (x,x^*)\in \Pi(X),~y^*\in S_{Y^*}\right\} \\ 
&=\sup\left\{\bigl|(x^*,y^*)(Q(x,0))\bigr|\colon \bigl((x,0),(x^*,y^*)\bigr)\in\Pi(X\oplus_1 Y)\right\} \leq \nu(Q)=1.
\end{align*}

Therefore, we have that
\begin{align*}
1= \|Q\circ E_X\| &\geq \|P_X\circ Q\circ E_Xx_1\| + \|P_Y\circ Q\circ E_Xx_1\| \\
&\geq \bigl|x_1^*(P_X\circ Q\circ E_Xx_1) + y^*_1(P_Y\circ Q\circ E_Xx_1)\bigr|=|(x_1^*, y_1^*)Q(x_1, 0)|=1.
\end{align*}
From this, we see that $P_Y\circ Q\circ E_Xx_1\neq 0$ since $\|P_X\circ Q\circ E_X\|=\|P_X\circ (Q-E_Y\circ T \circ P_X)\circ E_X\|<\frac{\eps}{2}$. Moreover, there exists a  scalar $\theta\in \K$ of modulus $1$ such that
\[
\|P_X\circ Q\circ E_Xx_1\| = \theta x_1^*(P_X\circ Q\circ E_Xx_1)\quad \text{~and~}\quad \|P_Y\circ Q\circ E_Xx_1\| = \theta y_1^*(P_Y\circ Q\circ E_Xx_1).
\]
Define the operator $S\in \mathcal{L}(X, Y)$ by
\[
Sx = P_Y\circ Q\circ E_Xx +\theta\frac{P_Y\circ Q\circ E_Xx_1}{\|P_Y\circ Q\circ E_Xx_1\|} x_1^*(P_X\circ Q\circ E_Xx) \qquad \bigl( x\in X\bigr).
\]
It is clear that $|y_1^*Sx_1| =1=\|S\|$. Finally,
\begin{align*}
\|S-T\|
&=\sup_{x\in B_X} \left\|P_Y\circ Q\circ E_Xx-Tx +\theta\frac{P_Y\circ Q\circ E_Xx_1}{\|P_Y\circ Q\circ E_Xx_1\|} x_1^*(P_X\circ Q\circ E_Xx)\right\|\\
&\leq \sup_{x\in B_X} \left(\left\|P_Y\circ Q\circ E_Xx-Tx\right\|+\left\|P_X\circ Q\circ E_Xx\right\|\right)\\
&= \sup_{x\in B_X} \|(Q-E_Y\circ T \circ P_X)(x,0)\|
\leq \|Q-E_Y\circ T \circ P_X\|
<\eps.\qedhere
\end{align*}
\end{proof}

\begin{corollary}\label{L1BPBZP} For arbitrary measures $\mu$ and $\nu$, the pair $(L_1(\mu),L_1(\nu))$ has the BPBZp.
\end{corollary}

\begin{proof}
Since $L_1(\mu)\oplus_1 L_1(\nu)$ is an $L_1$ space, $L_1(\mu)\oplus_1 L_1(\nu)$ has the BPBp-nu~(see \cite{KLM14}). Hence, $(L_1(\mu),L_1(\nu))$ has the BPBZp by Theorem \ref{thm1} and the fact that $n(L_1(\mu))=1$.
\end{proof}

Similarly to \cite[Theorem 2.3]{KLM16}, we obtain an $\ell_\infty$-sum version of Theorem \ref{thm1}.

\begin{theorem}\label{thm2}
Let $X$ and $Y$ be Banach spaces and suppose that $n(Y)=1$. If $X\oplus_\infty Y$ has the weak BPBp-nu with a function $\eta$, then $(X, Y)$ has the BPBZp with a function $\gamma\colon \eps\mapsto\eta\left(\frac{\eps}{4+\eps}\right)$ for $0<\eps<1$.
\end{theorem}

\begin{proof} Let $P_X$ and $P_Y$ be the canonical projections from $X\oplus_\infty Y$ to $X$ and from $X\oplus_\infty Y$ to $Y$ respectively, and $E_X$ and $E_Y$ be the canonical injections from $X$ to $X\oplus_\infty Y$ and from $Y$ to $X\oplus_\infty Y$ respectively.

For $0<\eps<1$, define $\eps' = \frac{\eps}{4+\eps}$ and assume $T\in \mathcal{L}(X, Y)$ with $\|T\|=1$ and $(x_0,y^*_0)\in S_X\times S_{Y^*}$ satisfy 
$$|y_0^*(Tx_0)|>1-\eta(\eps').$$

For $E_Y\circ T \circ P_X \in \mathcal{L}$($X\oplus_\infty Y$), from the Bishop-Phelps theorem, we see that
\begin{align*}
\nu(E_Y\circ T \circ P_X) &\leq \|E_Y\circ T \circ P_X\|=\|T\|= \sup \left\{|y^*(Tx)|\colon x\in S_X,~\text{norm attaining~}y^*\in S_{Y^*}\right\} \\ 
& = \sup \left\{ \bigl|(0,y^*)(E_Y\circ T \circ P_X(x,y))\bigr| \colon x\in S_{X},~(y,y^*)\in \Pi(Y)\right\} \\ 
&=\sup\left\{\bigl|(0,y^*)(E_Y\circ T \circ P_X(x,y))\bigr|\colon \bigl((x,y),(0,y^*)\bigr)\in\Pi(X\oplus_\infty Y)\right\}\\
&\leq \nu(E_Y\circ T \circ P_X).
\end{align*}
Hence, $\nu(E_Y\circ T \circ P_X)=\|E_Y\circ T \circ P_X\|=\|T\|=1$. We now take a norm attaining functional $z^*\in S_{Y^*}$ and its norm attaining point $z\in S_Y$ such that 
$$|z^*(Tx_0)| >1-\eta(\eps'),~z^*(z)=1~\text{and}~\|z^*-y_0^*\|<\eps/2$$
 from Bishop-Phelps theorem. Since it holds that 
$$\bigl((x_0,z),(0,z^*)\bigr)\in\Pi(X\oplus_\infty Y) \text{~and~}|(0, z^*)E_Y\circ T \circ P_X (x_0, z) |=|z^*(Tx_0)| >1-\eta(\eps'),$$
there exist $\bigl((x_1,y_1),(x_1^*,y_1^*)\bigr)\in\Pi(X\oplus_\infty Y)$ and $R\in \mathcal{L}(X\oplus_\infty Y)$ satisfying
$\nu(R)=\bigl|(x_1^*,y_1^*)R(x_1,y_1)\bigr|$, $\|R-E_Y\circ T \circ P_X \|<\eps'$, $\max\{\|x_1-x_0\|, \|y_1 -z\|\} <\eps'$ and $\|x_1^*\|+\|y_1^* - z^*\|<\eps'$.

Since $|\nu(R)-\nu(E_Y\circ T \circ P_X)|\leq \|R-E_Y\circ T \circ P_X\|<\eps'<1/5$ and $\nu(E_Y\circ T \circ P_X)=1=\|E_Y\circ T \circ P_X\|$, the operator $Q = {\nu(R)^{-1}}R\in\mathcal{L}(X\oplus_\infty Y)$ is well defined, and it holds that $\nu(Q) =1= \bigl|(x_1^*, y^*_1)Q(x_1, y_1)\bigr|$. Moreover, we have
\begin{align*}
\| Q - E_Y\circ T \circ P_X\|&\leq \| Q - R\|+\|R-E_Y\circ T \circ P_X\|\\
&< \|R\|\nu(R)^{-1}|1-\nu(R)|+ \eps' <(1+\eps')(1-\eps')^{-1}\eps'+\eps'=\eps/2.
\end{align*}

At this moment, we see that $x_1^*=0$ which implies $\|y_1^*\|=\|y_1\|=y_1^*(y_1) =1$ and $\|y_1^*-y_0^*\|<\eps$. This is followed by the easy observation $ \bigl|(0, y^*) U(x,y)\bigr|\leq \nu(U)\|y^*\|$ for any $\bigl((x, y),(x^*, y^*)\bigr)\in \Pi(X\oplus_\infty Y)$ and $U\in L(X\oplus_\infty Y)$.
Indeed, if $\|x_1^*\|\neq 0$, then
\begin{align*}
1&= \bigl|(x_1^*, y_1^*) Q(x_1, y_1) \bigr| \leq \left|\left(\frac{x_1^*}{\|x_1^*\|}, 0\right) Q(x_1, y_1) \right|\|x_1^*\| + \left|\left(0,y_1^*\right) Q(x_1,y_1) \right|\\
& \leq \left|\left(\frac{x_1^*}{\|x_1^*\|}, 0\right) (Q-E_Y\circ T \circ P_X)(x_1, y_1) \right|\|x_1^*\| + \nu(Q)\|y_1^*\|\\
&\leq \eps\|x_1^*\| +\|y_1^* \| < \|x_1^*\| + \|y_1^*\|=1
\end{align*}
which is a contradiction. This leads us to have $1=\nu(Q)=\bigl|(0, y_1^*) Q(x_1, y_1) \bigr|$.

We now see that $\|P_Y\circ Q\|=1$. Since we have that
$$1=\nu(Q)= |(0, y_1^*)Q(x_1, y_1)|=| y_1^* (P_Y\circ Q (x_1, y_1))|\leq \|P_Y\circ Q\|,$$
it is enough to show that $\|P_Y\circ Q(p,q)\|\leq 1$ for any fixed $(p,q) \in S_{X}\times S_{Y}$. Take $q^*\in S_{Y}$ satisfying $q^*(q)=1$, and define $H_{p,q} \in L(Y)$ by 
$$H_{p,q}(y)=P_Y\circ Q (q^*(y)p,y)\qquad (y\in Y).$$
Then, from the assumption $n(Y)=1$, we have
\begin{align*}
\|P_Y\circ Q(p,q)\|&=\|H_{p,q}q\|\leq \|H_{p,q}\|=\nu(H_{p,q}) \\ 
&\leq \sup\left\{|y^*(H_{p,q}y)|\colon (y,y^*)\in \Pi(Y) \right\}\\
&= \sup\left\{ |y^*(P_Y\circ Q(q^*(y)p,y))| \colon (y,y^*)\in \Pi(Y) \right\} \\
&\leq \sup\left\{ \bigl|y^* P_Y\circ Q(x,y)\bigr| \colon x\in B_X,\,~(y,y^*)\in \Pi(Y)\right\}\\
&=\sup\left\{ \bigl|(0,y^*) Q(x,y)\bigr| \colon \bigl((x,y),(0,y^*)\bigr)\in \Pi(X\oplus_\infty Y)\right\}\leq \nu(Q) =1.
\end{align*}

Therefore, we have that 
$$1=\|P_Y\circ Q\|=|y_1^*(P_Y\circ Q(x_1,y_1))|.$$
Moreover, we have that $\|x_1\|=1$. Otherwise, we see
\begin{align*}
|y_1^*( P_Y\circ Q (x_1,y_1))| &\leq \|x_1\|\left|y_1^*\left( P_Y\circ Q \left(\frac{x_1}{\|x_1\|},y_1\right)\right)\right| + (1-\|x_1\|)|y_1^*( P_Y\circ Q (0,y_1))|\\ &\leq \|x_1\| + (1-\|x_1\|)=1
\end{align*}
since $\|x_1-x_0\|<\eps$ implies $\|x_1\|>1-\eps$. Hence, we have $|y_1^*( P_Y\circ Q (0,y_1))|=1$ which contradicts to 
$$|y_1^*( P_Y\circ Q (0,y_1))|=|y_1^*( (P_Y\circ Q-T\circ P_X) (0,y_1))|\leq \|Q-E_Y\circ T\circ P_X\|<\eps/2.$$
For $x_2^*\in S_{X^*}$ with $x_2^*(x_1)=1$, define $S\in \mathcal{L}(X,Y)$ by
\[
Sx = P_Y\circ Q (x, x_2^*(x)y_1) \qquad \bigl(x\in X\bigr).
\]
Since it is clear that $|y_1^*Sx_1|=1=\|S\|$ and we have for any $x\in S_X$ that
\begin{align*}
\|Sx-Tx\|
&= \|P_Y\circ Q (x, x_2^*(x)y_1)-Tx\|\\
&\leq \|Q (x, x_2^*(x)y_1)-E_Y\circ T \circ P_X (x,x_2^*(x)y_1)\|\\
&\leq \|Q-E_Y\circ T \circ P_X\|<\eps/2,
\end{align*}
we finish the proof.
\end{proof}

\begin{corollary} The pair $(c_0,c_0)$ has the BPBZp.
\end{corollary}
\begin{proof}
Note that $c_0\oplus_\infty c_0$ is isometrically isomorphic to $c_0$ and it has  BPBp-nu (see \cite{GK13}). Hence $(c_0, c_0)$ has the BPBZp by Theorem~\ref{thm2}. 
\end{proof}

\begin{corollary} For compact metric spaces $K_1$ and $K_2$, the pair $(C(K_1),C(K_2))$ has the BPBZp in the real case.
\end{corollary}
\begin{proof} For the disjoint union $K_1\cup K_2$ with the disjoint union topology, it is clear that $K_1\cup K_2$ is a compact metrizable space and $C(K_1\cup K_2)$ is isometrically isomorphic to $C(K_1)\oplus_\infty C(K_2)$. Since $C(K_1\cup K_2)$ has the BPBp-nu in the real case by \cite[Corollary 3.3]{AGR14} and $n(C(K_2))=1$, we get that $(C(K_1),C(K_2))$ has the BPBZp from Theorem \ref{thm2}.
\end{proof}

The authors of the paper \cite{AGR14} introduced the concept of local compensation as a sufficient condition for a compact Hausdorff space $K$ to have the property that $C(K)$ has the BPBp-nu in the real case, and they also showed that every compact metric space admits local compensation (see \cite[Definition 2.1, Theorem 2.2, Theorem 3.2]{AGR14}). Recall that $C(\beta \N)$ is isometrically isomorphic to $\ell_\infty$ where $\beta \N$ is the Stone-\v{C}ech compactification of $\N$, and it is not known whether $\beta \N$ admits local compensation. In the following section, we will show that $(\ell_\infty, \ell_\infty)$ has the BPBZp in order to prove $\ell_\infty$ has the BPBp-nu in the real case, but this does not solve the question of the admision of local compensation for $\beta \N$.

\section{BPBZp for some classical Banach spaces}\label{section-BPBZp-pair-classical}
In this section, we aim to investigate the BPBZp for pairs of classical Banach spaces. We begin by examining pairs involving $\ell_1$ as the domain or $c_0$ and $\ell_\infty$ as the range, and conclude by studying pairs of the form $(L_1(\mu), Y)$.

\begin{definition}
Let $X$ be a Banach space and $Z$ be a subspace of $X^*$. We say that $X$ has \textit{$Z$-AHSp} if, for given $\eps>0$, there is $\eta(\eps)>0$ such that whenever $\{x_k\}_{k\in \mathbb{N}}\subset B_X$, $x^*\in S_Z$, and $a=(a_k)\in S_{\ell_1}$ with $a_k\geq 0$ for all $k\in\mathbb{N}$ satisfy
$$
\re x^*\left( \sum_{k=1}^{\infty} a_k x_k \right)>1-\eta(\eps),
$$
there are $A\subset \mathbb{N}$, $\{z_k\}_{k\in A}\subset S_X$ and $z^*\in S_{Z}$ such that
\begin{enumerate}
\item[(i)] $\sum_{k\in A} a_k>1-\eps$,
\item[(ii)] $\|z_k-x_k\|<\eps$ for all $k\in A$,
\item[(iii)] $z^*(z_k)=1$ for all $k\in A$,
\item[(iv)] $\|z^*-x^*\|<\eps$.
\end{enumerate}
\end{definition}

\begin{remark}\label{remAHSp} From their definitions, it is clear that if $X$ has the $X^*$-AHSp, then $X$ has the AHSp of Definition \ref{AHSp1}, and that $X$ has the AHSp for the pair $(X,X^*)$ of Definition \ref{AHSppair} if and only if $X^*$ has the $X$-AHSp. We also note that in order to prove that $X$ has the $Z$-AHSp, it is enough to check the conditions for every $a=(a_k)_{k\in \mathbb{N}}\in S_{\ell_1}$ with finite support with the same function $\eta$.
\end{remark}

We first give the following variant of \cite[Theorem 3.6]{ABGM13}.

\begin{theorem}\label{l1charac}
For a Banach space $Y$, the pair $(\ell_1,Y)$ has the BPBZp if and only if $Y$ has the $Y^*$-AHSp.
\end{theorem}

\begin{proof}Since the proof is almost the same as that of \cite[Theorem 3.6]{ABGM13}, we highlight only the key differences rather than providing the full details. In the case of \cite[Theorem 3.6]{ABGM13}, the sequence of vectors and a functional had to be chosen in $Y^*$ and $Y$ respectively, since a bilinear form on $\ell_1\times Y$ corresponds to an operator in $\mathcal{L}(\ell_1,Y^*)$. Since we are working with an operator in $\mathcal{L}(\ell_1,Y)$, we apply the same proof for the sequence of vectors and a functional in $Y$ and $Y^*$ respectively.
\end{proof}

The above characterization allows us to provide a list of spaces $Y$ with the $Y^*$-AHSp.

\begin{example}\label{exampleY*AHSP} A Banach space $Y$ has the $Y^*$-AHSp if it satisfies one of the following.
\begin{enumerate}
\item $Y$ is finite dimensional.
\item $Y$ is uniformly convex.
\item $Y=L_1(\mu)$ for an arbitrary measure $\mu$.
\end{enumerate}
\end{example}

\begin{proof}
(1) and (2): both $Y$ are reflexive and it is shown in \cite[Proposition 4.1 and 4.2]{ABGM13} that $Y_*$ has the AHSp for the pair $(Y_*,Y)$ where $Y_*$ is the predual of $Y$. Remark \ref{remAHSp} gives that $Y$ has the $Y^*~(\equiv Y_*)$-AHSp, as desired.

(3): This follows from Corollary \ref{L1BPBZP} and Theorem~\ref{l1charac}.
\end{proof}

Here is an example of $X$ without $X^*$-AHSp.

\begin{example}\label{CLexample} $C_0(L)$ does not have $C_0(L)^*$-AHSp for any infinite locally compact Hausdorff space $L$. 
\end{example}

\begin{proof} In order to prove this, we use the well-known identification $C_0(L)^*=M(L)$ where $M(L)$ is the space of scalar-valued regular Borel measures on $L$. For the reference, see \cite{Folland}. 

Assume that $C_0(L)$ has $C_0(L)^*$-AHSp with a function $\eta$. For $n_0\in \mathbb{N}$ such that $\frac{1}{n_0}<\eta\left(\frac{1}{2}\right)$, there exist disjoint compact sets $\Omega_1,\Omega_2,\ldots,\Omega_{n_0}\subset L$ since $L$ is infinite.

For each regular Borel probability measure $\mu_i$ on $\Omega_i$, define the  measure $\mu$ on $L$ by $\mu=\sum_{k=1}^{n_0}\frac{1}{n_0}\mu_k$, and take a function $f_i\in S_{C_0(L)}$ such that $f_i=0$ on $\Omega_i$ and $f_i=1$ on $\bigcup_{k\neq i} \Omega_k$. 

We note that 
$$\mu\left(\sum_{k=1}^{n_0}\frac{1}{n_0}f_k\right)=1-\frac{1}{n_0}>1-\eta\left(\frac{1}{2}\right).$$

From the assumption, there exist $A\subset \{1,2,\ldots,n_0\}$, $\{h_k\}_{k\in A}\subset S_{C_0(L)}$, and $\nu \in S_{M(L)}$, such that
\begin{enumerate}
\item[(i)] $\sum_{k\in A}\frac{1}{n_0} >\frac{1}{2}$,
\item[(ii)] $\|h_k-f_k\|<\frac{1}{2}$ for all $k\in A$,
\item[(iii)] $\nu(h_k)=1$ for all $k\in A$,
\item[(iv)] $\|\nu-\mu\|<\frac{1}{2}$.
\end{enumerate}
Write $$\nu=\sum_{k=1}^{n_0}\nu|_{\Omega_k}+\nu|_{L\setminus {\bigcup_{k=1}^{n_0} \Omega_k}}$$ where for a set $\Omega\subset L$, $\nu|_\Omega$ is the restriction of $\nu$ on $\Omega$. We see that 
$$1=\nu(h_i)\leq \sum_{k=1}^{n_0}\big|\nu|_{\Omega_k}(h_i)\big|+\big|\nu|_{L\setminus {\bigcup_{k=1}^{n_0} \Omega_k}}(h_i)\big|\leq \sum_{k=1}^{n_0}\big\|\nu|_{\Omega_k}\big\|+\big\|\nu|_{L\setminus {\bigcup_{k=1}^{n_0} \Omega_k}}\big\|=1$$
for every $i\in A$. This gives that $\big|\nu|_{\Omega_i}(h_i)\big|=\big\|\nu|_{\Omega_i}\big\|$. Since $|h_i|<\frac{1}{2}$ on $\Omega_i$, we have that $\left\|\nu|_{\Omega_i}\right\|=0$ for each $i\in A$. Hence, we have that 
$$\nu=\sum_{k\in A^c}\nu|_{\Omega_k}+\nu|_{L\setminus {\bigcup_{k=1}^{n_0} \Omega_k}}.$$
Therefore, we deduce the following which is the desired contradiction.
\begin{align*}
\|\nu-\mu\|
&=\sum_{k\in A}\big\|\mu|_{\Omega_k}\big\|+\sum_{k\in A^c}\big\|(\nu-\mu)|_{\Omega_k}\big\|+\big\|(\nu-\mu)|_{L\setminus {\bigcup_{k=1}^{n_0} \Omega_k}}\big\|\\
&\geq \sum_{k\in A}\frac{1}{n_0}\left\|\mu_k\right\|>\frac{1}{2}.\qedhere
\end{align*}
\end{proof}

The following is a direct consequence of Theorem \ref{l1charac} and Example \ref{CLexample}, using Theorems \ref{thm1} and \ref{thm2} (as both $\ell_1$ and $C_0(L)$ have numerical index $1$). It will be generalized for more general Lebesgue spaces in Corollary \ref{KLMstrengthen}, but we present it here for the sake of completeness.

\begin{corollary}\label{l1bpbpnu}
For an infinite locally compact Hausdorff space $L$, the pair $(\ell_1,C_0(L))$ does not have BPBZp. In particular, both $\ell_1\oplus_1 C_0(L)$ and $\ell_1\oplus_\infty C_0(L)$ do not have the BPBp-nu.
\end{corollary}

For the complex Banach space $X=\ell_1\oplus_\infty c_0$, we are going to show that the set of numerical radius attaining operators is dense in $\mathcal{L}(X)$. This space $X$ is the first known `natural' Banach space (i.e.\ constructed without somehow ``artificial'' renormings) which fails the BPBp-nu but such that the numerical radius attaining operators are dense in $X$ (see \cite{KLM14,KLM16} where such examples were constructed using renormings).

\begin{proposition}\label{prop:nra-dense-complex-l1c0}
For the complex Banach space $X=\ell_1\oplus_\infty c_0$, the set of numerical radius attaining operators is dense in $\mathcal{L}(X)$. 
\end{proposition}

\begin{proof} Let $P_{\ell_1}$ and $P_{c_0}$ be the canonical projections from $X$ to $\ell_1$ and from $X$ to $c_0$, respectively. For $k\in \mathbb{N}$, we also denote by $P_k$ and $E_k$ the canonical projection from $X$ to $\ell_1\oplus_\infty \ell_\infty^k$ and canonical injection from $\ell_1\oplus_\infty \ell_\infty^k$ to $X$ where $\ell_\infty^k$ is the $k$-dimensional space $\mathbb{C}^k$ with the supremum norm.

We first observe that the set of norm attaining operators is dense in $\mathcal{L}(X)$. Due to \cite[Lemma 2]{PayaSaleh} which states that the density of norm attaining operators from an arbitrary $\ell_1$-sum of Banach spaces into an arbitrary $\ell_\infty$-sum of Banach spaces is followed by that of those operators between each components, it is enough to show that the respective sets of norm attaining operators are dense in $\mathcal{L}(X,\ell_1)$ and $\mathcal{L}(X,c_0)$. In case of $\mathcal{L}(X,c_0)$, it is true since $c_0$ has property $(\beta)$ of Lindenstrauss (see \cite{Lindenstrauss63}). 

In order to prove the case of $\mathcal{L}(X,\ell_1)$, we recall that $\ell_1$ is $\C$-uniformly convex (see \cite{Globevnik75}). Indeed, the modulus of $\C$-convexity 
$$
\delta_\C(t) := \inf_{x,y \in S_{\ell_1}} \left\{ \sup_{|\lambda| = 1} \|x + \lambda t y\|-1 \right\}
$$
is strictly positive whenever $t>0$. 

For given $\eps>0$ and $T\in \mathcal{L}(X,\ell_1)$, we shall find a norm attaining operator $S\in \mathcal{L}(X,\ell_1)$ such that $\|T-S\|<\eps$. Without loss of generality, we assume that $\|T\|=1$, and take $x_0\in S_X$ such that 
$$\|Tx_0\|>1- \frac{\delta_\C(\eps/2)}{1+\delta_\C(\eps/2)}.$$

 From the density of vectors with finite supports in $c_0$, we may assume that the support of $P_{c_0}x_0$ is a subset of $\{1,2,3,\ldots,n\}$ for large $n\in \mathbb{N}$. Hence, we have $\|T\circ E_n\circ P_n x_0 \| >1- \frac{\delta_\C(\eps/2)}{1+\delta_\C(\eps/2)}$. Since
$$\|T\circ E_n\circ P_n x_0 + \lambda T\circ (\Id_X-E_n\circ P_n)z\|\leq 1$$
for every $z\in B_X$ and $\lambda\in \mathbb{C}$ of modulus $1$, we see that 
$$\left\|\frac{T\circ E_n\circ P_n x_0}{\|T\circ E_n\circ P_n x_0\|} + \lambda \frac{T\circ (\Id_X-E_n\circ P_n)z}{\|T\circ E_n\circ P_n x_0\|}\right\|\leq \frac{1}{\|T\circ E_n\circ P_n x_0\|}<1+\delta_\C\left(\frac{\eps}{2}\right).$$
Therefore, we have
$$\left\|\frac{T\circ (\Id_X-E_n\circ P_n)z}{\|T\circ E_n\circ P_n x_0\|}\right\|\leq \frac{\eps}{2}$$ from the definition of $\delta_\C$ and the convexity of the norm. Therefore, we get $$\|T-T\circ E_n\circ P_n\|=\|T\circ (\Id_X-E_n\circ P_n)\|\leq \frac{\eps}{2}.$$

Since $T\circ E_n$ is defined on $\ell_1\oplus_\infty \ell^n_\infty$ and $\ell_1\oplus_\infty \ell^n_\infty$ has the RNP, there exists a norm attaining operator $R\in \mathcal{L}(\ell_1\oplus_\infty \ell^n_\infty,\ell_1)$ such that $\|T\circ E_n-R\|<\frac{\eps}{2}$ (see \cite{Bourgain77}). It is clear that $S=R\circ P_n$ is the desired operator.

Finally, we observe that the set of numerical radius attaining operators is dense in $\mathcal{L}(X)$ by showing that every norm attaining operator in $\mathcal{L}(X)$ actually attains its numerical radius. Note that $n(X)=1$ (see \cite{MP00}) which means that the norm of the operator coincides with its numerical radius. Let $T\in \mathcal{L}(X)$ attain its norm at $x_0 \in S_X$, and take $x_0^*\in S_{X^*}$ such that $|x_0^*(Tx_0)|=\|T\|~(=\nu(T))$. 
Since $P_{\ell_1}x_0\in B_{\ell_1}$, by the convexity argument we may assume that $P_{\ell_1}x_0=e_{i}$ for some ${i}\in \mathbb{N}$ where $(e_k)_{k\in \mathbb{N}}$ is the canonical basis of $\ell_1$. Moreover, since $x_0^*\in S_{\ell_\infty\oplus_1 \ell_1}$ and $Tx_0\in \|T\|S_{\ell_1\oplus_\infty c_0}$, we may assume that $x_0^*$ is of the form $x_0^*=(u^*,0)$ for some extreme point of $B_{\ell_\infty}$ or $x_0^*=(0,e_j)$ for some $j\in \mathbb{N}$. In both cases, it holds that
$$x_0\in \aconv \{x\in B_{X}\colon x_0^*(x)=1\}$$
where $\aconv A$ is the absolute convex hull of a set $A$. Therefore, there exists $x_1\in B_X$ such that \[x_0^*(x_1)=1\quad \text{~and~} \quad |x_0^*(Tx_1)|=\|T\|.\qedhere\]
\end{proof}

We rewrite \cite[Proposition 3.5]{CKLM15} with our new notion.

\begin{theorem}[{\cite[Proposition 3.5]{CKLM15}}] \label{thm:c0newAHSp} For a Banach space $X$, the pair $(X,c_0)$ has the BPBZp if and only if $X^*$ has $X$-AHSp.
\end{theorem}

By the preceding theorem (or \cite[Proposition 3.5]{CKLM15}),  $(\ell_1,c_0)$ does not have BPBZp since $\ell_\infty$ does not have the $\ell_1$-AHSp (see \cite[Proposition 4.8]{ABGM13}). 

We now focus on pairs of the form $(X,\ell_\infty)$. We first get a necessary condition for $(X,\ell_\infty)$ to have the BPBZp.

\begin{theorem}\label{thm:inftyc0}
Let $X$ be a Banach space. If the pair $(X,\ell_\infty)$ has the BPBZp with a function $\eta$, then the pair $(X,c_0)$ has the BPBZp with the function $\eps\mapsto\eta\left(\frac{\eps}{2}\right)$ for $0<\eps<1$.
\end{theorem}

\begin{proof}
For $n\in \mathbb{N}$, let $P_{n}\colon \ell _{\infty }\rightarrow \ell _{\infty }^{n}$ be the canonical projection onto the first $n$ coordinates, where we identify $\ell _{\infty }^{n}$ as a subspace of $c_{0}$ (and thus of $\ell _{\infty }$). To simplify notation, we shall also denote the analogous projection from $c_{0}$ to $\ell _{\infty }^{n}$ by $P_{n}$.

For given $0<\eps<1$, assume that $T\in \mathcal{L}(X,c_0)$ with $\|T\|=1$ and $(x_0,y^*_0)\in S_X\times S_{\ell_1}$ satisfy 
$$|y^*_0(Tx_0)|>1-\eta\left(\frac{\eps}{2}\right).$$
Take $n_0\in \mathbb{N}$ such that 
$$\left\|P_{n_0}\circ T\right\|>1-\frac{\varepsilon}{2}~{and}~\left|y^*_0\left(P_{n_0}\circ Tx_0\right)\right|>1-\eta\left(\frac{\eps}{2}\right).$$
From the assumption, we have $(x_1,y_1^*)\in S_X\times S_{\ell_\infty^*}$ and $Q\in \mathcal{L}(X,\ell_\infty)$ such that 
$$\|Q\|=|y_1^*(Qx_1)|=1,~\|x_1-x_0\|<\frac{\varepsilon}{2},~\|y_1^*-y^*_0\|<\frac{\varepsilon}{2}\quad \text{and} \quad \left\|Q-\frac{P_{n_0}\circ T}{\|P_{n_0}\circ T\|}\right\|<\frac{\varepsilon}{2}.$$
Since it holds that
$$\left\|Q-P_{n_0}\circ T\right\|\leq \left\|Q-\frac{P_{n_0}\circ T}{\|P_{n_0}\circ T\|}\right\|+\left\|\frac{P_{n_0}\circ T}{\|P_{n_0}\circ T\|}-P_{n_0}\circ T\right\|<\eps.$$
we have
\begin{align*}
1
&=|y_1^*(Qx_1)|\leq\left|y_1^*\left(P_{n_0}\circ Qx_1\right)\right|+\left|y_1^*\left((\Id_{\ell_\infty}- P_{n_0})\circ Qx_1\right)\right|\\
&\leq \|P_{n_0}^*y_1^*\|+\left|y_1^*\left((\Id_{\ell_\infty}- P_{n_0})\circ (Q-P_{n_0}\circ T) x_1\right)\right|\\
&\leq \|P_{n_0}^*y_1^*\|+\left\|(\Id_{\ell_\infty}- P_{n_0})^*y_1^*\right\|\eps\\
&\leq \|P_{n_0}^*y_1^*\|+\left\|(\Id_{\ell_\infty}- P_{n_0})^*y_1^*\right\|=1.
\end{align*}
This implies $P_{n_0}^*y_1^*=y_1^*$ which means $y_1^*\in \ell_1$. Hence, for the operator 
$$S=P_{n_0}\circ Q + (\Id_{c_0}- P_{n_0})\circ T,$$ we deduce 
\begin{align*}
&|y_1^*(Sx_1)|=|y_1^*(P_{n_0}\circ Sx_1)|=|y_1^*(P_{n_0}\circ Qx_1)|=|y_1^*(Qx_1)|=1 \quad\text{and}\\
&\|S-T\|= \|P_{n_0}\circ Q-P_{n_0}\circ T\|\leq \|Q-P_{n_0}\circ T\|<\eps.
\end{align*}
which finishes the proof.
\end{proof}

\begin{corollary}
Let $X$ be  a Banach space. 
\begin{enumerate}
\item If the pair $(X,\ell_\infty)$ has the BPBZp, then $X^*$ has the $X$-AHSp. 
\item The pair  $(\ell_1,\ell_\infty)$ does not have the BPBZp.
\end{enumerate}
\end{corollary}

\begin{proof} These are direct consequences of Theorems \ref{thm:c0newAHSp} and \ref{thm:inftyc0}, and \cite[Proposition 4.8]{ABGM13}. Note that (2)  can also be obtained using Corollary \ref{l1bpbpnu}.
\end{proof}

We now get a sufficient condition for the pair $(X,\ell_\infty)$ to have the BPBZp. In \cite[Corollary 3.4]{ABGM13}, the following property on a Banach space $X$ is shown to be sufficient to get that $X^*$ has the $X$-AHSp: for every $\eps > 0$ there is $\eta(\eps) >0$ such that for every $x_0\in S_X$ there exists $x_1 \in S_X$ satisfying
\begin{enumerate}
\item[(i)] $\|x_1-x_0\| <\eps$,
\item[(ii)] if $x_0^*\in S_{X^*}$ satisfies $\re x_0^*(x_0)>1-\eta(\eps)$, then there exists $x_1^*\in S_{X^*}$ such that 
$$\|x_1^*-x_0^*\| < \eps ~\text{and}~x_1^*(x_1)=1.$$
\end{enumerate}

We say that $X^*$ has the \textit{$X$-AHp} if the aforementioned property holds. It was shown in \cite{ABGM13} that uniformly smooth spaces, finite-dimensional spaces, and the space $K(H)$ of compact operators on a Hilbert space $H$ all satisfy this property (see the proofs of \cite[Propositions 4.1, 4.2, 4.7]{ABGM13}). We now show that if a Banach space $X$ has this property, then the pair $(X, \ell_\infty)$ satisfies the BPBZp. 

\begin{theorem}\label{thm:AHpimpliesBPBZp}
Let $X$ be a Banach space. If $X^*$ has the $X$-AHp, then the pair $(X,\ell_\infty)$ has the BPBZp.
\end{theorem}

\begin{proof} 
Before we provide the proof, we remark that condition $(ii)$ of the $X$-AHp can be strengthened by taking $x_0^* \in B_{X^*}$, provided that the function $\eta$ is suitably relaxed.  We assume that $X^*$ has $X$-AHp witnessed by this function $\eta$ instead of the original one.

We use the canonical identification between $\mathcal{L}(X,\ell_\infty)$ and $\left[\bigoplus\nolimits_{i\in \N} X^*\right]_\infty$, the direct sum of countably many copies of $X^*$ endowed with the supremum norm. Under this identification, we denote $R\in \mathcal{L}(X,\ell_\infty)$ by $(R_i)\in \left[\bigoplus\nolimits_{i\in \N} X^*\right]_\infty$. Furthermore, for a subset $\Omega \subset \mathbb{N}$, $P_\Omega\colon \ell_\infty\rightarrow \ell_\infty(\Omega) \subset \ell_\infty$ denotes the canonical projection onto the coordinates indexed by $\Omega$.

 For given $0<\eps<1$, take $\delta(\eps)>0$ such that 
$$2\delta(\eps)<\min\left\{\eta\left(\eps\right),\eps\right\}.$$

Let $T\in S_{\mathcal{L}(X,\ell_\infty)}$, $y_0^*\in S_{\ell_\infty^*}$ and $x_0 \in S_X$ satisfy 
$$|y_0^*(Tx_0)|>1-\frac{\delta(\eps)^2}{2},$$
and, without loss of generality, we assume that $|y_0^*(Tx_0)|=y_0^*(Tx_0)$ by multiplying a suitable constant of modulus $1$ to $T$. 

By Bishop-Phelps-Bollob\'as theorem, we take $y_1^*\in S_{\ell_\infty^*}$ and $y_1\in S_{\ell_\infty}$ such that 
$$y_1^*(y_1)=1,~\|y_1^*-y_0^*\|<\delta(\eps)~\text{and}~\left\|y_1-Tx_0\right\|<\delta(\eps).$$

Set $A=\{i\in \mathbb{N}\colon |y_1(i)|>1-\delta(\eps)\}$ and define $y_2\in S_{\ell_\infty}$ by 
$$y_2(i)=\frac{y_1(i)}{|y_1(i)|}~\text{for}~i\in A~\text{and}~y_2(i)=y_1(i)~\text{for}~i\in A^c.$$

From the convexity of the norm, it is easy to see that $y_1^*(y_2)=1$. Moreover, we have that 
\begin{align*}
1
&=y_1^*(P_Ay_2)+y_1^*((\Id_{\ell_\infty}-P_A)y_2)=y_1^*(P_Ay_2)+y_1^*((\Id_{\ell_\infty}-P_A)^2y_2)\\
&\leq \|P_A^*y_1^*\|+\|(\Id_{\ell_\infty}-P_A)^*y_1^*\|(1-\delta(\eps))\leq \|P_A^*y_1^*\|+\|(\Id_{\ell_\infty}-P_A)^*y_1^*\|=1.
\end{align*}

Therefore, we get 
$$P_A^*y_1^*=y_1^* \text{~and~}y_1^*(P_Ay_2)=1.$$

 Since $\left\|y_2-Tx_0\right\|<2\delta(\eps)$, we have 
$$\left|y_2(i)-T_ix_0\right|<2\delta(\eps)<\eta\left(\eps\right) \text{~for~}i\in A,$$
and this implies
$$\re \overline{y_2(i)}T_ix_0>1-\eta\left(\eps\right) \text{~for~} i\in A.$$

We apply the definition of $X$-AHp to find $x_1\in S_{X}$ and $z_i^*\in S_{X^*}$ for each $i\in A$ such that 
$$\|x_1-x_0\|<\eps,~\left\|z_i^*-\overline{y_2(i)}T_i\right\|<\eps \text{~and~}z_i^*(x_1)=1.$$

We now construct an operator $S \in \mathcal{L}(X,\ell_\infty)$ by 
$$S_i=y_2(i)z_i^* ~\text{for}~i\in A,~\text{and}~S_i=T_i ~\text{for}~i\in A^c.$$

Since $P_A\circ S x_1=P_A y_2$ and $P_A^*y_1^*=y_1^*$, it is immediate that $1=y_1^*(P_A y_2)=y_1^*(S x_1).$

Finally, we get
\begin{align*}
\|S-T\|
&=\sup_{i\in A} \|S_i-T_i\|=\sup_{i\in A} \|y_2(i)z_i^*-T_i\|<\eps.
\end{align*}
which finishes the proof.
\end{proof}
We now present a family of examples satisfying this property.

\begin{proposition}\label{C0AHP} 
For a locally compact Hausdorff space $L$, $C_0(L)^*$ has the $C_0(L)$-AHp. 
\end{proposition}

\begin{proof}
For given $0<\eps<1$,  take $\gamma>0$ such that $\sqrt{2\gamma}+2\gamma<\eps$. We shall prove that $\eta(\eps)=\gamma^2$ is the desired function in the definition of $C_0(L)$-AHp.

For given $f_0\in S_{C_0(L)}$, set
\begin{align*}
V=\left\{t\in L\colon |f_0(t)|\geq 1-\frac{\eps}{4}\right\}\quad\text{and}\quad U=\left\{t\in L\colon |f_0(t)|> 1-\frac{\eps}{2}\right\}.
\end{align*}

Using Urysohn's lemma, take $\phi\in S_{C_0(L)}$ such that 
$$0\leq \phi\leq 1,~\phi(t)=1 ~\text{for}~t\in V,~\phi(t)=0~\text{for}~t\in U^c,$$
and define 
\begin{align*}
\psi(t)&=\frac{f_0(t)}{|f_0(t)|}\phi(t)~\text{for}~t\in U~\text{and}~\psi(t)=0~\text{for}~t\in U^c ~\text{and}\\
f_1(t)&=\psi(t)+(1-\phi(t))f_0(t)\in S_{C_0(L)} \text{~for~}t\in L.
\end{align*}

Since $|f_1-f_0|=|\psi-\phi f_0|$, we have $|f_1-f_0|=|\phi|\left|\frac{f_0}{|f_0|}-f_0\right|\leq |1-|f_0||<\frac{\eps}{2}$ on $U$ and $|f_1-f_0|=0$ on $U^c$ which implies $\|f_1-f_0\|<\eps$.

We now show that if a scalar-valued regular Borel measure $\mu \in S_{C_0(L)^*}~(=S_{M(L)})$ satisfies $$\re \int_Lf_0d\mu=\re\mu(f_0)>1-\gamma^2,$$ 
then there exists $\nu \in S_{C_0(L)^*}~(=S_{M(L)})$ such that $\|\mu-\nu\|<\eps$ and $\nu(f_1)=1$. In order to prove it, we first identify each $\varphi\in M(L)$ as 
$$d\varphi=h_\varphi d |\varphi|$$
for some Borel measurable function $h_\varphi$ such that $|h_\varphi|=1$ on $L$ using Radon-Nikod\'ym theorem (see  \cite[Proposition 3.13]{Folland}).

Set $O=\left\{t\in L\colon \re f_0(t)h_\mu(t) \geq 1-\gamma\right\}$. Note that $O \subset V$, and it holds that
\begin{align*}
\left|1-\frac{\overline{ f_0(t)h_\mu(t)}}{|f_0(t)|}\right|
&\leq \left(\left|1-\re \frac{\overline{ f_0(t)h_\mu(t)}}{|f_0(t)|}\right|^2+\left|\text{Im~}\frac{\overline{ f_0(t)h_\mu(t)}}{|f_0(t)|}\right|^2\right)^{\frac12}\\
&\leq \left(\left|1-\re \frac{\overline{ f_0(t)h_\mu(t)}}{|f_0(t)|}\right|^2+1-\left|\re \frac{\overline{ f_0(t)h_\mu(t)}}{|f_0(t)|}\right|^2\right)^{\frac12}\\
&<\sqrt{\gamma^2+1-(1-\gamma)^2}=\sqrt{2\gamma}
\end{align*}
for every $t\in O$. We also see that 
\begin{align*}
1-\gamma^2
&<\re \int_Lf_0 d\mu=\int_{O}\re f_0 h_\mu d|\mu|+\int_{O^c}\re f_0 h_\mu d|\mu|\\
&\leq \int_{O}1d|\mu|+(1-\gamma)\int_{O^c}1d|\mu|=\|\mu\|-\gamma\int_{O^c}d|\mu|.
\end{align*}

Hence, we get $|\mu|\left(O^c\right)<\gamma$, and so
$$1-\gamma^2-\gamma< \int_{O}\re f_0 h_\mu d|\mu|.$$

For the restriction $|\mu|\big|_O$ of $|\mu|$ on $O$, define the measure $\nu \in M(L)$ by 
$$d\nu=\frac{1}{|\mu|(O)}\overline{\psi}d |\mu|\big|_O.$$
Note that $\overline{\psi}=\frac{\overline{f_0 }}{\left|f_0 \right|}$ on $O$.
 Finally, we see that, for arbitrary $g\in S_{C_0(L)}$,
\begin{align*}
|\mu(g)-\nu(g)|
&\leq \left|\int_O g h_\mu d|\mu|-\int_O g \frac{1}{|\mu|(O)}\overline{\psi} d|\mu|\right|+|\mu|\left(O^c\right)\\
&< \left|\int_O \left(1-\frac{1}{|\mu|(O)}\overline{\psi}\overline{h_\mu}\right) g h_\mu d|\mu|\right|+\gamma \\
&\leq \left|\int_O \left(1-\overline{\psi}\overline{h_\mu}\right) g h_\mu d|\mu|\right|+\left|\int_O \left(1-\frac{1}{|\mu|(O)}\right) g\overline{\psi} d|\mu|\right|+\gamma \\
&<\sqrt{2\gamma}+2\gamma<\eps\\
\end{align*}
and
\[\nu(f_1)=\int_L f_1 d\nu =\int_L f_1\frac{1}{|\mu|(O)}\overline{\psi}d|\mu|\big|_O=\frac{1}{|\mu|(O)}\int_O \psi\overline{\psi}d\left|\mu\right|=1.\qedhere\]
\end{proof}

It is worth noting that Proposition~\ref{C0AHP}  strengthens \cite[Proposition 4.4]{ABGM13} which proves that $C_0(L)^*$ has the $C_0(L)$-AHSp. As a consequence of Proposition~\ref{C0AHP} and Theorem~\ref{thm:AHpimpliesBPBZp}, we obtain the following.

\begin{corollary}\label{C0BPBZP} 
For every locally compact Hausdorff space $L$, the pair $(C_0(L),\ell_\infty)$ has the BPBZp.
\end{corollary}

Our next aim is to study the BPBZp for pairs of the form $(L_1(\mu),Y)$.

\begin{theorem}\label{L1thm}
Let $\mu$ be a measure such that $L_1(\mu)$ is infinite dimensional and $Y$ be a Banach space. 
\begin{enumerate}
\item If the pair $(L_1(\mu),Y)$ has the BPBZp, then $Y$ has the $Y^*$-AHSp.
\item If $Y$ has the $Y^*$-AHSp and the RNP, then the pair $(L_1(\mu),Y)$ has the BPBZp.
\end{enumerate} 
\end{theorem}

\begin{proof}Since the statements (1) for a general measure and (2) for a $\sigma$-finite measure are followed by simple modifications of proofs of \cite[Theorem 2.4 and 2.6]{ABCGKLM14} respectively, we only comment the differences. In their proofs,  the sequences of vectors and a functional had to be chosen in $Y^*$ and $Y$ respectively since a bilinear form on $L_1(\mu)\times Y$ corresponds to an operator in $\mathcal{L}(L_1(\mu),Y^*)$. Since we are working with an operator in $\mathcal{L}(L_1(\mu),Y)$, we apply the same arguments with the sequence of vectors and a functional in $Y$ and $Y^*$ respectively. 

Note that if a Banach space $Y$ has the RNP and the $Y^*$-AHSp with a function $\eta$, then the pair $(L_1(\mu), Y)$ has the BPBZp with the same function $\eta$ for every $\sigma$-finite measure $\mu$ according to the modification of the proof of \cite[Theorem 2.6]{ABCGKLM14}.

In order to prove (2) for a general measure $\mu$, we need a modification of the proof of \cite[Proposition 2.1]{CKLM14}. Suppose that $T\in S_{\mathcal{L}(L_1(\mu),Y)}$, $f_0\in S_{L_1(\mu)}$ and $y^*_0\in S_{Y^*}$ satisfy 
$$\left|y^*_0(Tf_0)\right|>1-\eta(\eps).$$

 For a norming sequence $(g_n)_{n=1}^\infty\subset S_{L_1(\mu)}$ of $T$ (that means $\lim_{n\to \infty} \|Tg_n\|=\|T\|$), take countably many measurable sets $E_n$ with finite measures such that the closed linear span $\overline{\text{span}}\{ \chi_{E_k}\colon k\in \mathbb{N} \}$ of characteristic functions $\chi_{E_k}$ on $E_k$ contains $\{f_0\}\cup\{g_k\colon k\in \mathbb{N}\}$.

 Set $E= \bigcup_{k} E_k$ and $X = \{ f\chi_E \colon f\in L_1(\mu)\}$. Note that $X$ is isometrically isomorphic to $L_1(\mu|_E)$, and that, for the canonical projection $P\colon L_1(\mu) \longrightarrow X$ defined by $Pf := f\chi_E$ for $f\in L_1(\mu)$, $\|f\| = \|Pf\| + \|(\Id_{L_1(\mu)}-P)f\|$ for every $f\in L_1(\mu)$. For convenience, we also consider $P$ as an operator from $L_1(\mu)$ to $L_1(\mu|_E)$ if it is needed.

Define $T_1=T\circ P$ , then we see that $\|T_1\|=1$ and $|y^*_0(T_1 f_0)| > 1- \eta(\eps)$. Since $\mu|_E$ is $\sigma$-finite, there exist $S_1\in S_{\mathcal{L}(X,Y)}$, $f_1\in S_{X}$ and $y^*_1\in S_{Y^*}$ such that 
$$|y^*_1(S_1f_1)|=1,~\|T_1-S_1\|< \eps,~\|y_0^*-y_1^*\|<\eps\text{~and~}\|f_0-f_1\|< \eps.$$
Then, $S := S_1\circ P + T\circ (\Id_{L_1(\mu)}-P)$ is the desired operator in $S_{\mathcal{L}(L_1(\mu),Y)}$. Indeed, it holds that
$$\|Sf\|\leq \|S_1\circ Pf\|+\|T\circ (\Id_{L_1(\mu)}-P)f\|\leq \|Pf\|+\|(\Id_{L_1(\mu)}-P)f\|=\|f\|$$
for every $f\in L_1(\mu)$, and 
\[|y^*_1(Sf_1)|=|y^*_1(S_1f_1)|=1\quad \text{ and }\quad \|T-S\|=\|T_1-S_1\|< \eps.\qedhere\]
\end{proof}

We have the following consequences. Note that (2) of the following result was also given in Corollary \ref{uniformlysmoothBPBpcor}.

\begin{corollary}\label{KLMstrengthen}~
\begin{enumerate}
\item For a locally compact Hausdorff space $L$ and a measure $\mu$, if both $L_1(\mu)$ and $C_0(L)$ are infinite dimensional, then the pair $(L_1(\mu),C_0(L))$ fails the BPBZp. In particular, the spaces $L_1(\mu)\oplus_1 C_0(L)$ and $L_1(\mu)\oplus_\infty C_0(L)$ fail the BPBp-nu.
\item For every measures $\mu$ and $\nu$ and every $1<p<\infty$, the pair $(L_1(\mu),L_p(\nu))$ has the BPBZp.
\end{enumerate}
\end{corollary}

\begin{proof}(1) is a direct consequence of (1) of Theorem \ref{L1thm} and Example \ref{CLexample}, (2) is also a direct consequence of (2) of Theorem \ref{L1thm} and Proposition \ref{exampleY*AHSP}.
\end{proof}

Let us comment that it is shown in \cite{KLM16} that there exists a compact Hausdorff space $S$ such that both $L_1[0,1]\oplus_1 C(S)$ and $L_1[0,1]\oplus_\infty C(S)$ fail the BPBp-nu. This follows from a result of J.\ Johnson and J.\ Wolfe \cite{JW82} which asserts that there exists a compact Hausdorff space $S$ such that the set of norm attaining operators is not dense in $\mathcal{L}(L_1[0,1],C(S))$. Corollary \ref{KLMstrengthen}.(1) is a strengthening of that result in \cite{KLM16}.

\section{A sufficient condition for the BPBp-nu}\label{section:sufficient-condition-BPBpnu}

Our main goal in this section is to prove that the real $\ell_\infty$ space has the BPBp-nu. To this end, we introduce a property which is a sufficient condition for a Banach space $X$ to have the BPBp-nu under certain additional assumptions.

\begin{definition} A Banach space $X$ is said to have \emph{property (nu)} if, for every $\eps>0$, there exists $\eta(\eps)>0$ such that whenever $T\in S_{\mathcal{L}(X)}$, $x_0\in S_X$ and $x_0^*\in S_{X^*}$ satisfy $$\re x_0^*(x_0)>1-\eta(\eps)\quad \text{ and } \quad \left|x_0^*Tx_0\right|=1,$$ then there exist $x_1^*\in S_{X^*}$, $x_1\in S_X$ and $S\in S_{\mathcal{L}(X)}$ such that  $$\left|x_1^*Sx_1\right|=x_1^*(x_1)=1\quad \text{ and } \quad \max\{\|x_1-x_0\|,\|x_1^*-x_0^*\|,\|S-T\| \}<\eps.$$
\end{definition}

We first show that for a Banach space $X$ with numerical index $1$ such the pair $(X,X)$ has the BPBZp, property (nu) is a sufficient condition for $X$ to have the BPBp-nu.

\begin{theorem}\label{nuimpliesBPBPnu} Let $X$ be a Banach space with $n(X)=1$. If the pair $(X,X)$ has the BPBZp and $X$ has property (nu), then $X$ has the BPBp-nu.
\end{theorem}

\begin{proof}
Suppose that $(X,X)$ has the BPBZp witnessed by a function $\eta$ and that $X$ has property (nu) witnessed by a function $\gamma$.

 For $0<\eps<1$, define $\eps' = \min\left\{\frac{1}{2}\gamma\left(\frac{\eps}{2}\right), \frac{\eps}{2}\right\}$, and assume $T\in \mathcal{L}(X)$ and $(x_0,x^*_0)\in \Pi(X)$ satisfy  
$$\nu(T)=1\text{~and~} |x_0^*(Tx_0)|>1-\eta(\eps').$$

 Since $\|T\|=\nu(T)$, there exist $R\in \mathcal{L}(X)$ and $(z,z^*)\in S_X\times S_{X^*}$ satisfying 
$$|z^*(Rz)|=1=\|R\| \quad\text{~and~}\quad \max\{\|T-R\|,\|x_0-z\|,\|x_0^*-z^*\|\}<\eps'$$
from the assumption that $(X,X)$ has the BPBZp with a function $\eta$. This gives that 
$$\re z^* (z) \geq \re x_0^*(x_0)-\|x_0-z\|-\|x_0^*-z^*\|>1-2\eps'\geq 1-\gamma\left(\frac{\eps}{2}\right).$$

Hence, by property (nu), we have $S\in \mathcal{L}(X)$ and $(x_1,x_1^*)\in \Pi(X)$ satisfying
$$\left|x_1^*(Sx_1)\right|=1=\|S\| \quad\text{~and~}\quad \max\{\|S-R\|,\|x_1-z\|,\|x_1^*-z^*\|\}<\frac{\eps}{2}.$$

These give $\max\{\|S-T\|,\|x_1-x_0\|,\|x_1^*-x_0^*\|\}<\eps$ which finishes the proof.
\end{proof}

On the other hand, property (nu) is a necessary condition for a Banach space $X$ with numerical index $1$ to have the BPBp-nu.

\begin{theorem}\label{basicnu} Let $X$ be a Banach space with $n(X)=1$. If $X$ has the BPBp-nu, then $X$ has property (nu). In particular, under the conditions that $n(X)=1$ and $(X,X)$ has the BPBZp, property (nu) is equivalent to BPBp-nu.
\end{theorem}

\begin{proof} Suppose that $X$ has property (nu) witnessed by a function $\eta$, and, for $0<\eps<1$, take $\gamma(\eps)>0$ such that $\gamma(\eps)< \min\left\{\frac{1}{2}\eta\left(\frac{\eps}{2}\right), \frac{\eps}{2}\right\}$.

 Assume $T\in S_{\mathcal{L}(X)}$, $x_0\in S_X$ and $x_0^*\in S_{X^*}$ satisfy 
$$\re x_0^*(x_0)>1-\frac{\gamma(\eps)^2}{2} \text{~and~}\left|x_0^*Tx_0\right|=1.$$

 By Bishop-Phelps-Bollob\'as theorem, there exist $z\in S_X$ and $z^*\in S_{X^*}$ such that
$$\re z^*(z)=1~\text{and}~\max\{\|x_0-z\|,\|x_0^*-z^*\|\}<\gamma(\eps).$$

This gives that
$$|z^*Tz|\geq |x_0^*Tx_0|-\|x_0-z\|-\|x_0^*-z^*\|>1-2\gamma(\eps) >1-\eta\left(\frac{\eps}{2}\right).$$

Hence, there exist $x_1^*\in S_{X^*}$, $x_1\in S_X$ and $S\in S_{\mathcal{L}(X)}$ such that
$$|x_1^*Sx_1|=x_1^*(x_1)=1~\text{and}~\max\{\|x_1-z\|,\|x_1^*-z^*\|,\|S-T\| \}<\frac{\eps}{2}.$$

These give $\max\{\|S-T\|,\|x_1-x_0\|,\|x_1^*-x_0^*\|\}<\eps$ which finishes the proof.
\end{proof}

We do not know whether the condition $n(X)=1$ in the statement of Theorem \ref{basicnu} can be removed.

\begin{question}Does a Banach space $X$ have property (nu) whenever $X$ has the BPBp-nu?
\end{question}
We present some basic examples of spaces having property (nu).

\begin{example} The following spaces $X$ have property (nu).
\begin{enumerate}
\item $X=L_1(\mu)$ for every measure $\mu$.
\item $X=c_0$. 
\item $X=\mathcal{H}$ for a Hilbert space $\mathcal{H}$.
\end{enumerate}
\end{example}
\begin{proof}
(1) and (2) follow by Theorem \ref{basicnu} and  the facts that $X$ has the BPBp-nu (see \cite{GK13, KLM14}) and that $n(X)=1$ (see \cite{KaMaPa}).

(3) can be shown using the micro-transitivity of the norm of a Hilbert space (see \cite[Page 4]{DKLM20}). Indeed, let $\delta$ be the modulus of convexity of $\mathcal{H}$, and, for given $\eps>0$, assume $T\in S_{\mathcal{L}(\mathcal{H})}$, $x_0\in S_{\mathcal{H}}$ and $y_0\in S_{\mathcal{H}}$ satisfy 
$$\re \langle x_0,y_0\rangle >1-\delta(\eps) \text{~and~} \left|\langle Tx_0, y_0\rangle \right|=1.$$

By the definition of the modulus of convexity, we get $\|x_0-y_0\|<\eps$. Hence, by the micro-transitivity of the norm, there exists an isometry $R\in \mathcal{L}(\mathcal{H})$ such that 
$$x_0=Ry_0\quad \text{~and~}\quad \|\Id_{\mathcal{H}}-R\|=\|x_0-y_0\|<\eps.$$ 
Hence, we have that $\|T-T\circ R\|<\eps$, $\langle y_0,y_0\rangle =1$ and $\left|\langle T\circ Ry_0, y_0\rangle\right| =1$. 
\end{proof}

We now prove that the all the real spaces $C_0(L)$ have property (nu).

\begin{theorem}\label{prop-real-CK}
Let $L$ be a locally compact Hausdorff topological space. For the real space $C_0(L)$, the following holds:
For given $0<\varepsilon<1/4$, if $f_0\in S_{C_0(L)}$, $\mu \in S_{C_0(L)^*}$ and $T\in S_{\call(C_0(L))}$ satisfy
$$\mu(f_0)>1-\frac{\varepsilon^2}{50}\quad \text{~and~}\quad|\mu(Tf_0)|=1,$$
then there exist $f_1\in S_{C_0(L)}$ and $\nu\in S_{C_0(L)^*}$ such that
$$|\nu(Tf_1)|=\nu(f_1)=1,\quad \|f_1-f_0\|<\varepsilon,\quad \|\nu-\mu\|<\varepsilon.$$
In particular, every real $C_0(L)$ has property (nu).
\end{theorem}

In order to prove Theorem \ref{prop-real-CK}, we need the following lemma which is a version for $C_0(L)$ of \cite[Lemma 2.5]{AGR14}. Note that \cite[Lemma 2.5]{AGR14} was stated for $C(K)$ spaces (where $K$ compact), but the proof works for arbitrary locally compact Hausdorff spaces. Hence, we omit the proof.

\begin{lemma}[{\cite[Lemma 2.5]{AGR14}}]\label{lem:norm1C0L}
Let $L$ be a locally compact Hausdorff space and $f$ be an element of a real $C_0(L)$ space with $\|f\|=1$. Then,  a regular Borel measure $\mu\in S_{C_0(L)^*}$ satisfy $\int_L f\, d\mu=1$ if and only if 
$$
|\mu|((\{f=1\}\cap P)\cup (\{f=-1\}\cap N))=1
$$
where $P$ and $N$ are positive and negative parts in a Hahn decomposition of $\mu$.
\end{lemma}

\begin{proof}[Proof of Theorem \ref{prop-real-CK}]
Let $\xi=\varepsilon/5$ and set
$$V= \left\{ t: |f_0(t)|\ge 1-\frac{\xi}{3}\right\}\quad\text{and}\quad U= \{ t \colon |f_0(t)|\le 1-\xi\}.$$

From the assumption, we have
\[1-\frac{\varepsilon^2}{50}\le \int_{V^c} f_0d\mu + \int_{V} f_0d\mu \le \left(1-\frac\xi3\right)|\mu|(V^c)+ |\mu|(V)=1-\frac{\xi}3|\mu|(V^c).\]

Hence, we get $|\mu|(V^c)\le \frac{3\varepsilon}{10}<1$, which implies that $V$ is not empty.

For any subset $A\subset L$, we denote 
$$A^+=\{t\in A\colon f_0(t)\geq 0\}\quad\text{and}\quad A^-=A\backslash A^+.$$

Let $P_\mu$ and $N_\mu$ be the positive and negative parts of $\mu$ repectively in a Hahn decomposition of $\mu$. Then, we have
\begin{align*} 1-\frac{\epsilon^2}{50}&\le \mu(f_0) = \int_{P_\mu^+\cup N_\mu^-} f_0 d\mu + \int_{(P_\mu^+\cup N_\mu^-)^c} f_0 d\mu \\
&=\int_{P_\mu^+\cup N_\mu^-} |f_0| d|\mu| - \int_{(P_\mu^+\cup N_\mu^-)^c}|f_0| d|\mu|\\
&\le |\mu|(P_\mu^+\cup N_\mu^-).
\end{align*}
This gives $|\mu|((P_\mu^+\cup N_\mu^-)^c)\le \frac{\varepsilon^2}{50}$, and so we get
\begin{align*}
|\mu|(V\cap (P_\mu^+\cup N_\mu^-))
&= |\mu|(V)- |\mu|(V\setminus (P_\mu^+\cup N_\mu^-))\\
&\ge |\mu|(V)- |\mu|((P_\mu^+\cup N_\mu^-)^c)\\
&\ge 1-\frac{3\varepsilon}{10}-\frac{\varepsilon^2}{50}>0.
\end{align*}

 Define the regular Borel measure $\nu$ by \[\nu = \frac{1}{ |\mu|(V\cap ({P_\mu^+\cup N_\mu^-}))}\mu|_{V\cap ({P_\mu^+\cup N_\mu^-})}.\] 
Then, it holds that
\begin{align*}
\|\nu-\mu\| & \le \left(\frac{1}{ |\mu|(V\cap ({P_\mu^+\cup N_\mu^-}))}-1\right)|\mu|(V\cap ({P_\mu^+\cup N_\mu^-})) + |\mu|\left((V\cap ({P_\mu^+\cup N_\mu^-}))^c\right)\\
&\le 2\left(1-|\mu|(V\cap ({P_\mu^+\cup N_\mu^-}))\right)\le \frac{6\varepsilon}{10}+\frac{2\varepsilon^2}{50}<\varepsilon.
\end{align*} 

By the assumption $|\mu(Tf_0)|=1$ and Lemma~\ref{lem:norm1C0L}, we see that
$cTf_0=1$ $|\mu|$-a.e. on $P_\mu$ and $cTf_0=-1$ $|\mu|$-a.e. on $N_\mu$ for some constant $c$ of modulus $1$.
Since $|\nu|$ is absolutely continuous with respect to $|\mu|$, we see $cTf_0=1$ $|\nu|$-a.e. on $P_\mu$ and $cTf_0=-1$ $|\nu|$-a.e. on $N_\mu$. From the construction of $\nu$, $P_\mu$ and $N_\mu$ are also a Hahn decomposition for $\nu$, and it leads us to get $|\nu(Tf_0)|=1$.

We now construct $f_1\in S_{C_0(L)}$. Using Urysohn's lemma, take $\phi\in S_{C_0(L)}$ such that 
$$0\leq \phi\leq 1,~\phi(t)=1 ~\text{for}~t\in V,~\phi(t)=0~\text{for}~t\in U,$$
and define functions $\psi,g\in S_{C_0(L)}$ by
\begin{align*}
\psi(t)&=\frac{f_0(t)}{|f_0(t)|}\phi(t)~\text{for}~t\in U^c~\text{and}~\psi(t)=0~\text{for}~t\in U \text{~and}\\
g(t)&=\psi(t)+(1-\phi(t))f_0(t)~\text{for}~t\in L.\end{align*}

From Lemma~\ref{lem:norm1C0L} and the constructions of $\nu$ and $g$, we have that 
$$\nu(g)=1=\|g\|\text{~and~}\|g-f_0\|<\xi.$$

 Moreover, it also holds that $\sign(f_0(t))=\sign(g(t))$ for every $t\in L$ where $\sign(\theta) = \frac{\theta}{|\theta|}$ for each nonzero $\theta\in \mathbb{R}\setminus\{0\}$ and $\sign(0)=0$. Hence, the following function $f_1$ is well defined:
$$f_1(t)=\begin{cases}
\max\{f_0(t), g(t)\},\quad &\text{if }\sign(f_0(t))\geq 0\\
\min\{f_0(t), g(t)\},\quad &\text{if }\sign(f_0(t))< 0
\end{cases}.$$

Clearly, it holds that
$$\|f_1\|=1,~\|f_1-f_0\|\leq \|g-f_0\|<\varepsilon~\text{and}~\nu(f_1)=\nu(g)=1.$$
In order to finish the proof, we check that $|\nu(Tf_1)|=1$. This follows from the inequality $\|2f_0-f_1\|\leq 1$ which is a consequence of $|f_1|\geq |f_0|$ and $\sign(f_0(t))=\sign(f_1(t))$, and the equality $f_0=\frac{1}{2}\left((2f_0-f_1)+f_1\right)$.
\end{proof}

From Theorems \ref{nuimpliesBPBPnu}, \ref{prop-real-CK}, and Corollary \ref{C0BPBZP}, we deduce the main result of this section.

\begin{corollary}\label{cor:real-ell-infty-BPBp-nu}
The real Banach space $\ell_\infty$ has the BPBp-nu.
\end{corollary}

As noted in Sections \ref{section:intro} and \ref{section:bpbzp-and-bpbpnu}, it remains an open question whether the Stone-\v{C}ech compactification $\beta \mathbb{N}$ of $\mathbb{N}$ admits local compensation. Since $C(\beta \mathbb{N})$ is isometrically isomorphic to $\ell_\infty$, an affirmative answer to this question would immediately yield Corollary \ref{cor:real-ell-infty-BPBp-nu} in an alternative way, while a negative answer would solve question (b) of Section 4.3 of \cite{AGR14} in the negative.

\begin{question}Does the Stone-\v{C}ech compactification $\beta \N$ of $\N$ admit local compensation?
\end{question}

\section{BPBZp and BPBp-nu for compact operators}\label{section:cpt}
Numerous authors have investigated versions of the Bishop-Phelps-Bollob\'as type properties for compact operators, paralleling the problem of approximating compact operators by those which attain their norm (or numerical radius). This section is devoted to providing several remarks regarding our results for compact operators. To maintain brevity, we omit most of the details as the proofs follow by analogous arguments. We begin by recalling the definitions of these properties for compact operators for the sake of completeness.

\begin{definition}
~

\begin{itemize}
\item A pair $(X, Y)$ of Banach spaces is said to have the \textit{Bishop-Phelps-Bollob\'as property for compact operators} (\textit{BPBp for compact operators} for short) if, for every $\varepsilon > 0$, there exists $\eta(\varepsilon) > 0$ such that whenever $T \in \mathcal{K}(X,Y)$ with $\|T\| = 1$ and $x_0 \in S_X$ satisfy 
\[
\|T x_0\| > 1 - \eta(\varepsilon),
\]
there exist $x_1 \in S_X$ and $S \in \mathcal{K}(X,Y)$ such that 
\[
\|S\| = \|S x_1\| = 1, \quad \|x_1 - x_0\| < \varepsilon, \quad \text{and} \quad \|S-T\| < \varepsilon.
\]

\item A pair $(X, Y)$ of Banach spaces is said to have the \textit{Bishop-Phelps-Bollob\'as-Zizler property for compact operators} (\textit{BPBZp for compact operators} for short) if, for every $\varepsilon > 0$, there exists $\eta(\varepsilon) > 0$ such that whenever $T \in \mathcal{K}(X,Y)$ with $\|T\| = 1$, $y_0^* \in S_{Y^*}$, and $x_0 \in S_X$ satisfy 
\[
|y_0^*(T x_0)| > 1 - \eta(\varepsilon),
\]
there exist $y_1^* \in S_{Y^*}$, $x_1 \in S_X$, and $S \in \mathcal{K}(X,Y)$ such that 
\[
\|S\| = |y_1^*(S x_1)| = 1, \quad \|x_1 - x_0\| < \varepsilon, \quad \|y_1^* - y_0^*\| < \varepsilon, \quad \text{and} \quad \|S-T\| < \varepsilon.
\]

\item A Banach space $X$ is said to have the \textit{(resp. weak) Bishop-Phelps-Bollob\'as property for numerical radius for compact operators} (\textit{(resp. weak) BPBp-nu for compact operators} for short) if, for every $0<\eps<1$, there exists $\eta(\eps)>0$ such that
whenever $T\in \mathcal{K}(X)$ and $(x_0, x_0^*)\in \Pi(X)$ satisfy  
$$\nu(T)=1 \text{~and~} |x_0^*Tx_0|>1-\eta(\eps),$$
there exist $S\in \mathcal{K}(X)$ and $(x_1, x_1^*)\in \Pi(X)$ such that
\[
~\quad \quad  \quad\nu(S) = |x_1^*Sx_1|=1~(\text{resp.}~\nu(S) = |x_1^*Sx_1|), \ \ \|S-T\|<\eps, \ \ \|x_1-x_0\|<\eps,\ \ \text{and} \ \ \|x_1^*- x_0^*\|<\eps.
\]
\item A Banach space $X$ is said to have \emph{property (nu) for compact operators} if for every $\eps>0$ there exists $\eta(\eps)>0$ such that whenever $T\in S_{\mathcal{K}(X)}$, $x_0\in S_X$ and $x_0^*\in S_{X^*}$, satisfy $$\re x_0^*(x_0)>1-\eta(\eps)\quad \text{ and } \quad \left|x_0^*Tx_0\right|=1,$$ then there exist $x_1^*\in S_{X^*}$, $x_1\in S_X$ and $S\in S_{\mathcal{K}(X)}$ satisfying that  $$\left|x_1^*Sx_1\right|=x_1^*(x_1)=1\quad \text{ and } \quad \max\{\|x_1-x_0\|,\|x_1^*-x_0^*\|,\|S-T\| \}<\eps.$$
\end{itemize}
\end{definition}
One can readily verify that all the results in Section \ref{section:bpbzp-and-bpbpnu} remain valid for compact operators. For the sake of brevity, we summarize them here without providing details.

\begin{theorem}\label{thm1cpt}
Let $X$ and $Y$ be Banach spaces. 
\begin{enumerate}
\item If $Y$ is uniformly smooth and the pair $(X,Y)$ has the BPBp for compact operators, then $(X,Y)$ has the BPBZp for compact operators.
\item If $n(X)=1$ and $X\oplus_1 Y$ has the weak BPBp-nu for compact operators with a function $\eta$, then $(X, Y)$ has the BPBZp for compact operators with the function $\gamma\colon \eps\mapsto\eta\left(\frac{\eps}{4+\eps}\right)$ for $0<\eps<1$. 
\item If $n(Y)=1$ and $X\oplus_\infty Y$ has the weak BPBp-nu for compact operators with a function $\eta$, then $(X, Y)$ has the BPBZp for compact operators with a function $\gamma\colon \eps\mapsto\eta\left(\frac{\eps}{4+\eps}\right)$ for $0<\eps<1$.
\end{enumerate}
\end{theorem}

By analogous proofs, the following well-known facts for general operators carry over to the case of compact operators:  (1) if $X$ is uniformly convex, then $(X, Y)$ has the BPBp for every Banach space $Y$ (see \cite{KL14a}); (2) all pairs of the form $(L_1(\mu), L_q(\nu))$ have the BPBp whenever $1 < q < \infty$ for arbitrary measures $\mu$ and $\nu$ (see \cite{CKLM14}); (3) if $L$ is a locally compact Hausdorff space and $Y$ is $\mathbb{C}$-uniformly convex (resp. uniformly convex), then the pair $(C_0(L), Y)$ has the BPBp for complex (resp. real) Banach spaces (see \cite{Acosta16}); (4) the space $L_1(\mu)$ has the BPBp-nu for every measure $\mu$ (see \cite{KLM14}). Furthermore, in contrast to the general operator setting, it is known that $C_0(L)$ has the BPBp-nu for compact operators for any locally compact Hausdorff space $L$ (see \cite{GMMR21}). Hence, we conclude the following from Theorem \ref{thm1cpt}.

\begin{corollary}
For the following Banach spaces $X$ and $Y$, the pair $(X,Y)$ has the BPBZp for compact operators.
\begin{enumerate}
\item $X=L_1(\mu)$ and $Y=L_1(\nu)$ for  arbitrary measures $\mu$ and $\nu$.
\item $X=L_p(\mu)$ and $Y=L_q(\nu)$ for $1\leq p \leq \infty$, $1<q<\infty$, and arbitrary measures $\mu$ and $\nu$.
\item $X=C_0(L)$ and $Y=L_q(\nu)$ for $1<q<\infty$, arbitrary measure $\nu$, and arbitrary locally compact Hausdorff space $L$.
\item $X=C_0(L)$ and $Y=C_0(M)$ for arbitrary locally compact Hausdorff spaces $L$ and $M$.
\end{enumerate}
\end{corollary}

With the exception of Theorem \ref{thm:AHpimpliesBPBZp}, every result in Section \ref{section-BPBZp-pair-classical} remains valid for compact operators via analogous arguments. In particular, we point out the characterization of spaces $Y$ such that $(L_1(\mu), Y)$ has the BPBZp for compact operators is provided by the $Y^*$-AHSp. In the proof of \cite[Theorem 2.6]{ABCGKLM14}, which underlies the proof of Theorem \ref{L1thm}, the RNP was required to represent a given operator as a countably vector-valued measurable function. However, since every compact operator admits such a representation, the RNP on $Y$ is no longer necessary in this context.

The following theorem and corollary are the compact operator versions of Theorems \ref{L1thm}, \ref{thm:c0newAHSp}, \ref{thm:inftyc0} and Corollary \ref{KLMstrengthen}.

\begin{theorem}\label{l1characcpt}
Let $X$ and $Y$ be Banach spaces.
\begin{enumerate}
\item Whenever $\mu$ is a measure such that $L_1(\mu)$ is infinite dimensional, the pair $(L_1(\mu),Y)$ has the BPBZp for compact operators if and only if $Y$ has the $Y^*$-AHSp.
\item The pair $(X,c_0)$ has the BPBZp for compact operators if and only if $X^*$ has $X$-AHSp.
\item If the pair $(X,\ell_\infty)$ has the BPBZp for compact operators, then so does the pair $(X,c_0)$.
\end{enumerate}
\end{theorem}

\begin{corollary}\label{KLMstrengthencpt}
For a locally compact Hausdorff space $L$ and a measure $\mu$, assume the spaces $L_1(\mu)$ and $C_0(L)$ are infinite dimensional.
\begin{enumerate}
\item The pair $(L_1(\mu),C_0(L))$ fails the BPBZp for compact operators.
\item The spaces $L_1(\mu)\oplus_1 C_0(L)$ and $L_1(\mu)\oplus_\infty C_0(L)$ fail the BPBp-nu  for compact operators.
\end{enumerate}
\end{corollary}

It is worth noting that Proposition \ref{prop:nra-dense-complex-l1c0} holds for compact operators. Similar to the case of general operators, the complex Banach space $\ell_1\oplus_\infty c_0$ is the first example constructed without the renormings which fails BPBp-nu for compact operators such that the set of numerical radius attaining compact operators is dense in the space of compact operators.
\begin{proposition}
For the complex Banach space $X=\ell_1\oplus_\infty c_0$, the set of numerical radius attaining compact operators is dense in $\mathcal{K}(X)$,  but $X$ fails BPBp-nu for compact operators. 
\end{proposition}

The same happens with the results of Section \ref{section:sufficient-condition-BPBpnu}. Regarding property (nu) for compact operators, the previous results can be straightforwardly adapted. Notably, since $C_0(L)$ satisfies the BPBp-nu for compact operators for every locally compact Hausdorff space $L$, it follows that $C_0(L)$ has property (nu) for compact operators in both the real and complex settings. We provide a summary of these results.

\begin{theorem}\label{nuimpliesBPBPnucpt} Let $X$ be a Banach space with $n(X)=1$.
\begin{enumerate}
\item If the pair $(X,X)$ has the BPBZp for compact operators and $X$ has property (nu) for compact operators, then $X$ has the BPBp-nu for compact operators.
\item If $X$ has the BPBp-nu for compact operators, then $X$ has property (nu) for compact operators. 
\item Whenever $(X,X)$ has the BPBZp for compact operators, property (nu) for compact operators is equivalent to the BPBp-nu for compact operators.
\end{enumerate}
\end{theorem}

\begin{example}\label{cptnuex} The following spaces $X$ have property (nu) for compact operators.
\begin{enumerate}
\item $X=L_1(\mu)$ for every measure $\mu$.
\item $X=\mathcal{H}$ for a Hilbert space $\mathcal{H}$.
\item $X=C_0(L)$ where $L$ is a locally compact Hausdorff space.
\end{enumerate}
\end{example}

\subsection*{Acknowledgements} S.K.\ Kim has been supported by the National Research Foundation of Korea(NRF) grant funded by the Korea government (MSIT) [NRF-2020R1C1C1A01012267].  M.\ Mart\'{\i}n has been partially supported by the grant PID2021-122126NB-C31 funded by MICIU/AEI/10.13039/501100011033 and ERDF/EU, by ``Maria de Maeztu'' Excellence Unit IMAG reference CEX2020-001105-M, funded by MICIU/AEI/10.13039/501100011033, and by Junta de Andaluc\'{\i}a: grant FQM-0185. \'O.\ Rold\'{a}n has been supported by MICIU/AEI/10.13039/501100011033 and ERDF/EU through the grants PID2021-122126NB-C31 and PID2021-122126NB-C33.

\end{document}